\documentclass[a4paper]{article}

\usepackage{amsmath,amsfonts,amssymb}
\usepackage{a4wide}
\usepackage{amsthm}
\usepackage{graphicx}
\usepackage{caption}

\DeclareMathOperator\dd{d}

\makeatletter
\renewcommand\paragraph{\@startsection{paragraph}{4}{\z@}%
            {-2.5ex\@plus -1ex \@minus -.25ex}%
            {1.25ex \@plus .25ex}%
            {\normalfont\normalsize\bfseries}}
\makeatother
\DeclareFontFamily{OMX} {MnSymbolE}{}
\DeclareFontShape{OMX}{MnSymbolE}{m}{n}{
  <-6> MnSymbolE5
  <6-7> MnSymbolE6
  <7-8> MnSymbolE7
  <8-9> MnSymbolE8
  <9-10> MnSymbolE9
  <10-12> MnSymbolE10
  <12-> MnSymbolE12}{}
\DeclareFontShape{OMX}{MnSymbolE}{b}{n}{
  <-6> MnSymbolE-Bold5
  <6-7> MnSymbolE-Bold6
  <7-8> MnSymbolE-Bold7
  <8-9> MnSymbolE-Bold8
  <9-10> MnSymbolE-Bold9
  <10-12> MnSymbolE-Bold10
  <12-> MnSymbolE-Bold12}{}

\DeclareSymbolFont{MnSyE} {OMX} {MnSymbolE}{m}{n}

\DeclareMathDelimiter{\llangle}{\mathopen}{MnSyE}{116}{MnSyE}{120}
\DeclareMathDelimiter{\rrangle}{\mathopen}{MnSyE}{121}{MnSyE}{125}

\title{Transformations and singularities of polarized curves}
\date{July 2018}
\author{Andreas Fuchs 
	}

\begin{document}
%Kommando fürs Verstecken von Inhalten
\newcommand{\ltz}{\lim_{t\rightarrow 0}}

\newcommand{\Com}{\mathbb C}

\newcommand{\Nat}{\mathbb N}
\newcommand{\Pro}{\mathbb P}
\newcommand{\R}{\mathbb R}

\newcommand{\cB}{\mathcal B}
\newcommand{\Fr}{\mathcal F}
\newcommand{\I}{\mathcal I}
\newcommand{\Li}{\mathcal L}
\newcommand{\Ord}{\mathcal O}
\newcommand{\cS}{\mathcal S}
\newcommand{\cU}{\mathcal U}
\newcommand{\cV}{\mathcal V}

\newcommand{\afr}{\mathfrak {a}}
\newcommand{\bfr}{\mathfrak {b}}
\newcommand{\cfr}{\mathfrak {c}}
\newcommand{\Cfr}{\mathfrak {C}}
\newcommand{\ffr}{\mathfrak {f}}
\newcommand{\chfr}{\hat{\mathfrak {c}}}
\newcommand{\fhfr}{\hat{\mathfrak {f}}}
\newcommand{\gfr}{\mathfrak {g}}
\newcommand{\so}{\mathfrak {so}}
\newcommand{\gli}{\mathfrak {gl}}
\newcommand{\nfr}{\mathfrak {n}}
\newcommand{\ort}{\mathfrak {o}}
\newcommand{\pro}{\mathfrak {p}}

\newcommand{\ipl}{\llangle} % inner product left
\newcommand{\ipr}{\rrangle} % inner product right
\newcommand{\Ipl}{\left\llangle} % inner product left bigger
\newcommand{\Ipr}{\right\rrangle} % inner product right bigger

\newcommand{\id}{t}
\newcommand{\Rmn}{\R^{n+2}_1}
\newcommand{\bs}{\backslash}
\newcommand{\bd}{\ell}

\newcommand{\gatr}[2]{#1\!\ltimes \! #2} % gauge transform by map 
\newcommand{\Gtr}[3]{#1 \ltimes_{\!#2} #3}

\newcommand{\Qend}{\mathfrak Q}   %----- ENDOMORPHISM RELATED TO POLARIZATION
\newcommand{\Qendc}{\mathfrak Q}   %----- ENDOMORPHISM RELATED TO POLARIZATION FOR CURVES
\newcommand{\Qho}{\mathcal Q}   %----- HOPF DIFFERENTIAL
\newcommand{\Qua}{Q}   %----- HOLOMORPHIC QUADRATIC DIFFERENTIAL
\newcommand{\Pol}{Q}   %----- HOLOMORPHIC QUADRATIC DIFFERENTIAL
\newcommand{\of}{\go} % 1-form associated to surface or curve
\newcommand{\Of}{\gO} % o-valued lift of it
\newcommand{\pf}{\psi} % symbol for pole form
\newcommand{\Pf}{\Psi} % symbol for lift of pole form
\newcommand{\ppf}{\xi} % pure pole form
\newcommand{\Ppf}{\boldsymbol{\Xi}} % orthogonal lift of pure pole form
\newcommand{\aof}{\psi} % arbitrary one form
\newcommand{\Aof}{\Psi} % lift of arbitrary one form
\newcommand{\baszer}{o}
\newcommand{\basinf}{\boldsymbol{\iota}}
\newcommand{\bastan}{\mathfrak t}
\newcommand{\basnor}{\mathfrak n}

\newcommand{\na}{\nabla}
\newcommand{\la}{\langle}
\newcommand{\ra}{\rangle}
\newcommand{\La}{\left\langle}
\newcommand{\Ra}{\right\rangle}
\newcommand{\LiSe}{\pounds}

\newcommand{\pr}[1]{{#1}^{\prime}}
\newcommand{\dpr}[1]{{#1}^{\prime\prime}}
\newcommand{\ti}[1]{\tilde{#1}}
\newcommand{\p}{\partial}
\newcommand{\ch}{{\hat c}}
\newcommand{\Ch}{{\hat C}}

\newcommand{\inp}[1]{\left(#1\right)}

\newcommand{\hyfr}{{g}}

\newtheorem{proposition}{Proposition}[section]
\newtheorem{theorem}[proposition]{Theorem}
\newtheorem{lemma}[proposition]{Lemma}
\newtheorem{corollary}[proposition]{Corollary}
\newtheorem{definition}[proposition]{Definition}
\newtheorem{defiandlemma}[proposition]{Definition and Lemma}
\newtheorem{remark}[proposition]{Remark}
\newtheorem{example}{Example}
\newcommand{\ga}{\alpha}
\newcommand{\gb}{\beta}
\newcommand{\gga}{\gamma}
\newcommand{\gd}{\delta}
\newcommand{\gD}{\Delta}
\newcommand{\gth}{\theta}
\newcommand{\vth}{\vartheta}
\newcommand{\vphi}{\varphi}
\newcommand{\gk}{\kappa}
\newcommand{\gl}{\lambda}
\newcommand{\gL}{\Lambda}
\newcommand{\go}{\omega}
\newcommand{\gO}{\Omega}
\newcommand{\gs}{\sigma}
\newcommand{\gS}{\Sigma}
\newcommand{\gep}{\epsilon}

\newcommand{\gG}{\Gamma}
\maketitle

\begin{abstract}
We study the limiting behaviour of Darboux and Calapso transforms of polarized curves in the conformal $n$-dimensional sphere, when the polarization has a pole of first or second order at some point. We prove that for a pole of first order, as the singularity is approached all Darboux transforms converge to the original curve and all Calapso transforms converge. For a pole of second order, a generic Darboux transform converges to the original curve while a Calapso transform has a limit point or a limit circle, depending on the value of the transformation parameter. In particular, our results apply to Darboux and Calapso transforms of isothermic surfaces when a singular umbilic with index $\frac{1}{2}$ or $1$ is approached along a curvature line.


%53A30 conformal geometry
%37K35 Lie-B#cklund and other transformations
%37K40 Soliton theory, asymptotic behavior of solutions
%58K45 Singularities of vector fields, topological aspects
\end{abstract}

\section{Introduction}
Transformations of surfaces play a central role in our present understanding of smooth and discrete differential geometry. Not only do they allow the construction of new surfaces of a given class from existing ones, but their existence reveals a great deal about the underlying (integrable) structure of the corresponding classes of surfaces, cf \cite{bur06}. For example, given a pseudospherical surface ($K=-1$), the B\"acklund transformation yields a two-parameter family of new pseudospherical surfaces by solving an integrable, first order partial differential equation (see \cite[\S 120]{eis60}). In \cite{ter00}, these B\"acklund transformations are shown to be generators of an infinite-dimensional transformation group of the space of pseudospherical surfaces and the relations among these generators are given by a permutability theorem discovered by Bianchi \cite[\S 257]{bia99} in the nineteenth century. In particular, it is shown that pseudospherical surfaces fall into the class of integrable systems. Similar developments have been achieved for other classes of surfaces such as minimal surfaces, constant mean curvature surfaces, constant Gaussian curvature surfaces (cf \cite{bob94}) and curved flats in symmetric spaces such as Darboux pairs of isothermic surfaces, cf \cite{cie95}, \cite{fer96} and \cite{bur97}.

In the discrete theory, the importance of transformations becomes even more apparent, as articulated in \cite{bob07} (see also \cite{bob08}): ``In this setting, discrete surfaces appear as two-dimensional layers of multidimensional discrete nets, and their transformations correspond to shifts in the transversal lattice directions. A characteristic feature of the theory is that all lattice directions are on equal footing with respect to the defining geometric properties. Due to this symmetry, discrete surfaces and their transformations become indistinguishable.''

The interplay between aspects of discrete and smooth differential geometry was explored in \cite{bur16} with the study of semi-discrete isothermic surfaces, introduced in \cite{mue13}. In analogy to the transformation theory of smooth isothermic surfaces (see \cite[Ch 8.6]{jer03}, \cite{bur06} or \cite{bur10}), the authors develop a notion of Christoffel, Darboux and Calapso transformations of polarized curves, that is, smooth curves equipped with a nowhere zero quadratic differential, called a polarization. They then show that semi-discrete isothermic surfaces are sequences of Darboux transforms of polarized curves. In line with the ideas of discrete differential geometry, permutability theorems of the transformations of polarized curves are shown to yield a corresponding transformation theory for semi-discrete isothermic surfaces.

We are interested in the class of smooth isothermic surfaces, classically characterized by the local existence of conformal curvature line coordinates, away from umbilics. For these surfaces, the transformation theory is defined only locally, that is, on simply connected surface patches on which regular nets of conformal curvature line coordinates exist (cf \cite{bur06}, \cite{jer03} and \cite{smy04}). A global transformation theory is still missing. In particular, it seems necessary to reconsider the definition of an isothermic surface. The classical definition makes no restriction whatsoever on umbilic points. But certain configurations of umbilics are an obstacle for a global definition of the transformations (see \cite[\S 5.2.20]{jer03}). One candidate for an alternative definition of an isothermic surface is to require the existence of a globally defined holomorphic quadratic differential whose trajectories (see \cite[Sect 5.5]{str84}) agree with the curvature lines of the surface on the complement of its umbilic set. But due to the Poincar\'e-Hopf theorem \cite{hop89}, this would exclude all surfaces homeomorphic to a sphere, for example the ellipsoid, cf \cite{smy04}. If we merely require the existence of a meromorphic quadratic differential, that includes topological spheres. Moreover, according to the local Carath\'eodory conjecture, the poles of such a meromorphic quadratic differential are at most of order two (see \cite{gui08}). 

This paper makes a first step towards a global transformation theory of isothermic surfaces for which an underlying meromorphic quadratic differential exists. A curvature line of an isothermic surface together with the restriction of such a meromorphic quadratic differential yields a polarized curve in the sense of \cite[Sect 2]{bur16}. Moreover, the transformations for isothermic surfaces descend to the corresponding transformations of the polarized curvature lines. Thus, in order to understand how the transforms of an isothermic surface behave when a pole of the underlying meromorphic quadratic differential is approached along a curvature line, we investigate how Darboux and Calapso transforms of polarized curves with singular polarizations behave when the singularity is approached. In particular, we investigate the limiting behaviour of Darboux and Calapso transforms of polarized curves where the polarization has a pole of first or second order. According to the local Carath\'eodory conjecture, these are the only poles that can occur on a curvature line of an isothermic surface. 

In Sect \ref{section:defi_trafos}, we give M\"obius geometric definitions of the Darboux and Calapso transformations of polarized curves in the conformal $n$-dimensional sphere $S^n$. Our definitions are formulated with a projective model of M\"obius geometry: we identify $S^n$ with the projectivization $\Pro (\Li^{n+1})$ of the light cone $\Li^{n+1}$ in $n+2$-dimensional Minkowski space. The M\"obius group can then be identified with the projective Lorentz group $\Pro O(\Rmn)$. Curves are described as maps into the projective light cone. For computations and in order to prove that our definitions agree with those of \cite{bur16}, we relate the projective model to its linearisation based on the action of $O(\Rmn)$ on $\Li^{n+1}$, where curves are described via their light cone lifts. We then give a definition of a polarization with a pole of first or second order. 

The central object in our definition of the transformations (Def \ref{defi:trafos}) are the primitives $\gG_p(\gl\of)$ of a family $(\gl\of)_{\gl\in\R}$ of 1-forms associated to a polarized curve. Here, the primitive of a 1-form $\pf$ with values in the Lie algebra $\pro\ort(\Rmn)$ of $\Pro O(\Rmn)$ is the unique map $t\mapsto\gG_p^{~t}(\pf)$ into $\Pro O(\Rmn)$ which satisfies
$$\dd\gG_p(\pf)=\gG_p(\pf)\pf,~~~\gG_p^{~p}(\pf)=id,$$
cf \cite{sha97}. When the polarization has a pole of first or second order at some point, the 1-forms $\gl\of$ associated to the polarized curve also have a pole at that point and hence the primitives are not defined there. Nevertheless, as we prove in Cor \ref{cor:maintool_sing_firstkind} and Prop \ref{prop:maintool_sing} in Sect \ref{section:maintool}, the primitives of a certain class of $\pro\ort(\Rmn)$-valued 1-forms with a pole of first order do have a limit in $\Pro End(\Rmn)\supset \Pro O(\Rmn)$ at the singular point. 

In Sect \ref{section:first_order_pole}, we use these results to show that when the polarization of a polarized curve has a pole of first order, every Calapso transform converges to some point and all Darboux transforms converge to the original curve.

In Section \ref{section:sec_order_pole}, we consider the case of a polarization with a pole of second order. We cannot apply the results of Sect \ref{section:maintool} here directly because the 1-forms $\gl\of$ associated to the polarized curve have poles of second order. We first have to do a singular gauge transformation to transform the 1-forms with poles of second order to 1-forms with a pole of first order. The behaviour of the Darboux and Calapso transforms in this case is more diverse. A generic Darboux transform still converges to the original curve as the singularity is approached, but there are Darboux transforms which do not converge. Moreover, the Calapso transform has a limit point or a limit circle, depending on the value of the transformation parameter $\gl$.

\textit{Acknowledgements:} The author would like to thank G. Szewieczek for fruitful discussions and S. Fujimori for his hospitality during a visit to Okayama University, where parts of this work were elaborated. Special thanks go to U. Hertrich-Jeromin for his numerous, valuable comments and his support as supervisor of the author's PhD thesis. This work has been supported by the Austrian Science Fund (FWF) and the Japan
Society for the Promotion of Science (JSPS) through the FWF/JSPS Joint Project grant I1671-N26
``Transformations and Singularities''.

%-------------------------------------------------------------------------------
%-------------------------------------------------------------------------------
\section{Darboux and Calapso transforms of polarized curves}\label{section:defi_trafos}
In this section we define Darboux and Calapso transforms of polarized curves, show that our definitions agree with those of \cite{bur16} and specify the goal of this paper: the study of the limiting behaviour of Darboux and Calapso transforms at points where the polarization has a pole of first or second order.

We use the projective model of M\"obius geometry (cf \cite[Sect 1.1.]{bur06} and \cite[Ch 1]{jer03}) and identify the conformal $n$-sphere $S^n$ with the projectivization $\Pro (\Li^{n+1})$ of the light cone in $n+2$-dimensional Minkowski space. In this way, the action of the M\"obius group on $S^n$ can be identified with the action of the projective Lorentz group $\Pro O(\Rmn)$ on $\Pro(\Li^{n+1})$. The projective Lorentz group is diffeomorphic to the group $O^+(\Rmn)$ of orthochronous Lorentz transformations via the diffeomorphism $\la \cdot\ra$ of taking the linear span. Its differential at the identity yields the isomorphism of Lie algebras
\begin{align}
\dd_{id} \la\cdot\ra:~\ort(\Rmn)&\rightarrow \pro\ort(\Rmn),\notag\\
v\wedge w&\mapsto v\wedge w+\R~id,\label{eq:iso_liealg}
\end{align}
where we further identify $\gL^2(\Rmn)$ with $\ort(\Rmn)$ via
$$(v\wedge w)(x)=\ipl v,x\ipr\, w-\ipl w,x\ipr \,v$$
for $v,w,x\in\Rmn$. Here $\ipl \cdot,\cdot\ipr$ denotes the Minkowski inner product on $\Rmn$. 
%Notationally, we distinguish an element $v\wedge w+\R~id$ of $\pro\ort(\Rmn)$ from the corresponding $v\wedge w\in \ort(\Rmn)$ by the summand $\R~id$.

A regular curve $\la c\ra$ in $S^n=\Pro(\Li^{n+1})$ is an immersion of some interval $(a,b)$ into $\Pro(\Li^{n+1})$, where the notation $\la c\ra$ indicates that such a map may also be described by pointwise taking the linear span of a light cone lift $c:(a,b)\rightarrow \Li^{n+1}$ of $\la c\ra$. For simplicity, we always assume $\la c\ra$ to be smooth. A polarized curve $(\la c\ra,Q)$ in $S^n$ is a regular curve $\la c\ra$ in $S^n$ together with a nowhere zero quadratic differential $Q$ on $(a,b)$. To a polarized curve $(\la c\ra,Q)$ we associate the $\pro\ort(\Rmn)$-valued 1-form $\of$ given by
\begin{equation}
\of=\Qend~c\wedge \dd c+\R~id,
\label{eq:defi_ass_oneform}
\end{equation}
where $c$ is any light cone lift of $\la c\ra$ and the function $\Qend$ is related to $Q$ via that lift $c$ by
$$Q=\Qend~\|\dd c\|^2.$$
The 1-form $\of$ is independent of the choice of lift $c$ of $\la c\ra$. 

Let $G$ be a Lie group with Lie algebra $\gfr$ and $\pf$ a $\gfr$-valued 1-form on $(a,b)$. Then for any $p\in (a,b)$, denote by 
$$\gG_{p}(\pf):~(a,b)\rightarrow G,~~t\mapsto \gG_{p}^{~t}(\pf),$$
the unique primitive that satisfies (cf \cite[Ch 3, Thm 6.1]{sha97})  
\begin{equation}
\dd\gG_{p}(\pf)=\gG_{p}(\pf)\pf,~~~~\gG_{p}^{~p}(\pf)=id.
\label{eq:propsgG1}
\end{equation}
From the defining properties \eqref{eq:propsgG1} of $\gG_{p}(\pf)$ and its uniqueness it follows readily that
\begin{equation}
\forall p,t\in(a,b):~~~~~\gG_{p}^{~t}(\pf)\gG_{t}(\pf)=\gG_{p}(\pf).
\label{eq:propsgG2}
\end{equation}
Therefore, the map $\gG^{~p}(\pf):(a,b)\ni t\mapsto \gG_t^{~p}(\pf)\in G$ is the composition of $\gG_{p}(\pf)$ with taking the inverse in $G$ and thus satisfies
$$\dd \gG^{~p}(\pf)=-\pf \gG^{~p}(\pf),~~~~\gG_p^{~p}(\pf)=id.$$
Under a gauge transformation 
$$\pf\mapsto \gatr g \pf:=g^{-1}\pf g+g^{-1}\dd g$$
 using a smooth map $g:(a,b)\rightarrow G$, the primitives $\gG_{p}(\pf)$ transform as
\begin{equation}
\gG_{p}(\pf)=g(p)\,\gG_{p}(\gatr g \pf)\,g^{-1}.
\label{eq:gauge_trafo_beha}
\end{equation}

\begin{definition}\label{defi:trafos}
Let $(\la c\ra,Q)$ be a polarized curve in $S^n$ with associated 1-form $\of$ given by \eqref{eq:defi_ass_oneform}.

For $\gl\in\R\bs\{0\}$, $p\in (a,b)$ and $\la \ch_p\ra\in S^n$, the curve
$$\la \ch\ra:=\gG^{~p}(\gl\of)\la \ch_p\ra$$
is called the $\gl$-\textit{Darboux transform} of $(\la c\ra,Q)$ with initial point $\la \ch(p)\ra=\la \ch_p\ra$.

For $\gl\in\R$ and $p\in (a,b)$, the curve
$$\la c_{\gl,p}\ra:=\gG_{p}(\gl\of)\la c\ra$$
is called the $\gl$-\textit{Calapso transform of $(\la c\ra,Q)$ normalized at} $p$.
\end{definition}
Due to \eqref{eq:propsgG2}, two $\gl$-Calapso transforms normalized at $p$ and $q$, respectively, differ by a M\"obius transformation. Moreover, the Darboux and Calapso transformations of a polarized curve $(\la c\ra,Q)$ in $S^n$ are invariant under M\"obius transformations of $S^n$ and reparametrizations of $\la c\ra$. We remark that our definition is somewhat asymmetric because the transform of a polarized curve is not again a polarized curve, but merely a curve. But as we do not study repeated transformations of polarized curves here, there is no necessity nor benefit in equipping the transforms of a polarized curve with a polarization. 

%We say that a polarized curve $(\la c\ra,Q)$ on $(a,b)$ is regular at $a$ if there is a polarized curve on some larger interval $(a-\gep,b)$, $\gep>0$, which agrees with $(\la c\ra,Q)$ upon restriction to $(a,b)$. Otherwise, we call it singular at $a$. 
This paper is devoted to the study of the limiting behaviour of the Darboux and Calapso transforms at $a$ in the case that $\la c\ra$ can be extended regularly to some $(a-\gep,b)$, but the polarization has a pole of first or second order at $a$ in the sense of
\begin{definition}\label{defi:pole}
A quadratic differential $Q:(a,b)\ni t\mapsto Q_t=Q(t)\dd t^2$ or a Lie algebra-valued 1-form $\pf:(a,b)\ni t\mapsto \pf_t=\pf(t)\dd t$ on $(a,b)$ has a pole of order $k\in\mathbb N$ at $a$ if the function $t\mapsto Q(t)(a-t)^k$ or $t\mapsto\pf(t)(a-t)^k$, respectively, has a smooth extension to $(a-\gep,b)$ for some $\gep>0$ with nonzero value at $a$.
\end{definition}
If $\la c\ra$ has a smooth extension to $(a-\gep,b)$, then the polarization $Q$ has a pole of first or second order at $a$ if and only if the 1-form $\go$ associated to $(\la c\ra,Q)$ has a pole of first or second order at $a$.

%The main difficulty, with which we deal with in Sect \ref{section:maintool}, is to study the limiting behaviour of primitives of such singular $\pro\ort(\Rmn)$-valued 1-forms at their singular points.

The above definition of a pole is invariant under those diffeomorphisms of $(a,b)$ to $(\ti a,\ti b)$ which extend to diffeomorphisms of some larger interval $(a-\gep,b)$ to $(\ti a-\ti\gep,\ti b)$ with $\gep,\ti\gep>0$, e.g. translations 
$$(a,b)\ni t\mapsto t-a\in (0,b-a).$$
Together with the reparametrization invariance of the Darboux and Calapso transformation, we can thus make the convenient choice $a=0$.
\vspace{0.1cm}\\

%In what follows, we exclude the trivial case $\gl=0$, in which all Darboux transforms are constants and all Calapso transforms agree with the original curve.

Our definitions of the transformations of a polarized curve are formulated entirely in the projective model of M\"obius geometry, where $\Pro O(\Rmn)$ acts on $\Pro (\Li^{n+1})$ and surfaces are described as immersions into $\Pro (\Li^{n+1})$. Especially in computations, a linearisation of this model proves useful, where $O^+(\Rmn)$ acts on the light cone $\Li^{n+1}$ and surfaces are described via their light cone lifts as immersions into $\Li^{n+1}$. Let $\pf$ be a $\pro\ort(\Rmn)$-valued 1-form on $(0,b)$ and $\Pf$ its \textit{orthogonal lift} defined by $\Pf+\R~id=\pf$ via the isomorphism \eqref{eq:iso_liealg}. Then, the $O^+(\Rmn)$-valued primitives $\gG_p(\Pf)$ satisfy
$$\forall p\in(0,b):~~~~~\dd \la \gG_p(\Pf)\ra=\la \gG_p(\Pf)\ra \dd \la\cdot\ra(\Pf)=\la \gG_p(\Pf)\ra \psi~~~~~\text{and}~~~~~\la \gG_p^{~p}(\Pf)\ra=id.$$
Thus, the maps $\la \gG_p(\Pf)\ra$ are exactly the primitives of $\pf$. More generally, for any real, nowhere vanishing function $h$ on $(0,b)$, we have
\begin{equation}
\forall p\in(0,b):~~~~~\gG_p(\pf)=\gG_p(\Pf+\R~id)=\la \gG_p(\Pf)\ra=\la h\, \gG_p(\Pf)\ra.
\label{eq:equality_of_prims}
\end{equation}
In particular, this identity allows to compute the primitives of a $\pro\ort(\Rmn)$-valued 1-form via arbitrary nowhere zero rescalings of the primitives of its orthogonal lift. 

The relation \eqref{eq:equality_of_prims} also enables us to show that our Def \ref{defi:trafos} agrees with the corresponding ones in \cite{bur16}. Namely, let $(\la c\ra,Q)$ be a polarized curve with associated 1-form $\go$ and orthogonal lift $\gO$. Using \eqref{eq:equality_of_prims}, it follows that with $\ch_p\in \la \ch_p\ra$ the lift $\ch=\gG^{~p}(\gl\Of)\ch_p$ of a $\gl$-Darboux transform $\la \ch\ra$ of $(\la c\ra,Q)$ as defined in Def \ref{defi:trafos} is a parallel section of the connection
$$\frac{D^\gl}{\dd t}=\frac{\dd}{\dd t}+\gl \Qend\,c\wedge \pr c~~~~\text{with}~~~~\Qend\,\|\dd c\|^2=Q,$$
and hence yields a $(-\tfrac \gl 2)$-Darboux transform\footnote{Our parameter $\gl$ is $-1/2$ times the parameter $t$ used in \cite{bur16}.} of $(\la c\ra,Q)$ in the sense of \cite[Def 2.5]{bur16}. Conversely, any parallel section $\ch$ of $\frac{D^\gl}{\dd t}$ yields a $\gl$-Darboux transform $\la \ch\ra$ according to our Def \ref{defi:trafos}.

Similarly, the $\gl$-Calapso transform normalized at some $p\in (0,b)$ up to M\"obius transformation corresponds to the $(-\tfrac \gl 2)$-Calapso transform defined in \cite[Def 3.4]{bur16} because, up to M\"obius transformation, their Calapso transformation $T^{-\gl/2}$ is precisely $\gG_p(\gl\Of)$.

Not surprisingly, the main advantage of the linearisation of the projective formalism is its linear structure. For example, integration of \eqref{eq:propsgG1} shows that the primitives of an $\ort(\Rmn)$-valued 1-form $\Pf$ on $(0,b)$ satisfy the useful integral equation in $\R^{(n+2)^2}$
\begin{equation}
\forall p,t\in(0,b):~~~~\gG_p^{~t}(\Pf)=id+\int_p^{t}\gG_p^{~\tau}(\Pf)\Pf(\tau)\dd\tau.
\label{eq:integral_eq}
\end{equation}

Its disadvantage, on the other hand, is its reliance on the choice of lifts. As we will see in Sect \ref{section:maintool}, when a $\pro\ort(\Rmn)$-valued 1-form $\pf$ and thus also its orthogonal lift $\Pf$ on $(0,b)$ have a pole, it may happen that the $O^+(\Rmn)$-valued primitives $\gG_p(\Pf)$ tend towards infinity as the pole is approached, while the corresponding $\Pro O(\Rmn)$-valued primitives $\gG_p(\pf)$ have limits in $\Pro (\R^{(n+2)^2})\supset \Pro O(\Rmn)$. If this is the case for the 1-form $\of$ associated to a polarized curve and its orthogonal lift $\gO$, then the lifts $\ch=\gG^{~p}(\gl\Of)\ch_p$ and $c_{\gl,p}=\gG_p(\gl\Of)c$ diverge as the singularity is approached,  while the corresponding maps $\la \ch\ra$ and $\la c_{\gl,p}\ra$ converge. In such a situation, a lift independent formalism is clearly beneficial.

\section{Main tools}\label{section:maintool}
In this Section, we study the limiting behaviour at $0$ of the primitives $\gG_{p}(\pf)$ of a certain class of $\pro\ort(\Rmn)$-valued 1-forms $\pf$ on $(0,b)$ with a pole of first order at $0$. In Subsection \ref{section:fixonedim}, we restrict to 1-forms of a particularly simple form, which we call \textit{pure pole forms}. For those, the primitives and their limiting behaviour at $0$ can be computed explicitly. In Subsection \ref{section:genoneform} we then relate the limiting behaviour of primitives of the more general pole forms to that of the pure pole forms. In view of our definition of the transformations of a polarized curve via the primitives of multiples $\gl\of$ of the associated 1-form $\of$ (Def \ref{defi:trafos}), these tools will provide the means to investigate the limiting behaviour of the transforms of a singular polarized curve in Sects \ref{section:first_order_pole} and \ref{section:sec_order_pole}.

\subsection{Primitives of pure pole forms}\label{section:fixonedim}
We call a $\pro\ort(\R^{n+2}_1)$-valued 1-form $\ppf$ on $(0,b)$ a \textit{pure pole form} if it is of the form
$$\ppf_t=-v\wedge w \frac{\dd t}{\id}+\R~id$$
for some $v,w\in \Rmn$ with $v\wedge w\neq 0$. Pure pole forms come in three types corresponding to the three signatures which the linear span $\la v,w\ra\subset \Rmn$ can have. Accordingly, we speak of Minkowski, degenerate and spacelike pure pole forms. 

In addition to the quadratic form $\|\cdot\|^2$ induced by the indefinite Minkowski inner product $\ipl \cdot,\cdot\ipr$, we introduce the standard, Euclidean, positive-definite norm $|\cdot|$ on $\R^{n+2}$ and $\R^{(n+2)^2}$. On $End(\Rmn)\simeq \R^{(n+2)^2}$, the norm $|\cdot|$ is submultiplicative and satisfies $| A^*|=|A|$ for $A^*$ the adjoint of $A\in End(\Rmn)$ with respect to the \textit{Minkowski} inner product on $\Rmn$. On $\pro\ort(\Rmn)$, we define $|\cdot|$ by the push-forward of $|\cdot|$ restricted to $\ort(\Rmn)\subset \R^{(n+2)^2}$ via the isomorphism \eqref{eq:iso_liealg}. Furthermore, for $w\in \Rmn$, we denote by $w^*$ the linear map $w^*:v\mapsto \ipl v,w\ipr$.

\begin{proposition}\label{prop:boundednonzerosimplesections}
Let $\ppf$ be a pure pole form on $(0,b)$ and $\Ppf$ its orthogonal lift. Then there is a constant $B\in\R$ such that
\begin{enumerate}
	\item if $\ppf$ is Minkowski, then
	\begin{align}
	\forall p\in(0,b)~\forall t\in (0,p]:~~~~&\left|\left(\frac{t}{p}\right)^\zeta\gG_{p}^{~t}(\Ppf)\right| < B,\label{eq:bound_sim_min}\\
	\forall p\in(0,b):~~~~&\lim_{t\rightarrow 0}\left(\frac{t}{p}\right)^\zeta\gG_{p}^{~t}(\Ppf)= \frac{v_+v_-^*}{\ipl v_+,v_- \ipr},\label{eq:lim_sim_min}
	%\\
	%&\lim_{t\rightarrow 0} \gG_{p}^t(\ga_t\dd t~id)=0&~&\forall p\in(0,b)~,\label{eq:prankprop_sim_min}
	\end{align}
	where $v_\pm\in\la v,w\ra\cap\Li^{n+1}$ are eigenvectors of $v\wedge w$ with eigenvalues $\pm\zeta$, $\zeta>0$;
	\item if $\ppf$ is spacelike, then
	\begin{equation}
	\forall p,t\in(0,b):~~~~\left|\gG_{p}^{~t}(\Ppf)\right| < B,~~~~~~~~~
	\label{eq:bound_sim_spa}
	\end{equation}
	but $\gG_p(\ppf)$ does not converge as $0$ is approached as it is a rotation in the plane $\la v,w\ra$ with speed increasing towards infinity;
	
	\item if $\ppf$ is degenerate, choose $v_0\in \la v,w\ra\cap \Li^{n+1}$ and $\ti w\in \la v,w\ra$ such that $v\wedge w=v_0\wedge \ti w$. Then
	\begin{align}
	~~~~~~~~~~~~~~\forall p,t\in(0,b):~~~~&\left|\frac{1}{\bd(p,t)}\gG_{p}^{~t}(\Ppf)\right|<B,\label{eq:bound_sim_deg}\\
~~~~~~~~~~~\forall p\in(0,b):~~~~&\lim_{t\rightarrow 0}\frac{1}{\bd(p,t)}\gG_{p}^{~t}(\Ppf)=-\frac{\|\ti w\|^2}2 v_0v_0^*,\label{eq:lim_sim_deg}
	\end{align}
	where $\bd:(0,b)^2\rightarrow \R$ is the symmetric function
	\begin{equation}
	\bd:(p,t)\mapsto 1+\left(\ln\left(\frac{t}{p}\right)\right)^2.
	\label{eq:defi_bound_ell}
	\end{equation}
\end{enumerate}
\end{proposition}
\begin{proof}

First we note that $\gG_{p}^{~t}(\Ppf)$ is given by the exponential
\begin{equation}
\gG_{p}^{~t}(\Ppf)=e^{-\int_{p}^t\frac{\dd \tau}{\tau}v\wedge w}=e^{-\ln\left(\frac{t}{p}\right)v\wedge w}.
\label{eq:sim_paralltr}
\end{equation}

\begin{enumerate}
\item For $\la v,w\ra$ Minkowski, we may choose $v_\pm\in\Li^{n+1}$ such that $\ipl v_+,v_- \ipr>0$ and $v\wedge w=v_-\wedge v_+$. Then $v_\pm$ are eigenvectors of $v\wedge w$ with eigenvalues $\pm\zeta=\pm\ipl v_+,v_-\ipr$ and $\la v_+,v_-\ra^\perp=ker(v\wedge w)$. Using that
$$(s~v_-\wedge v_+)^k=(s\zeta)^k~\left(\frac{v_+v_-^*}\zeta\right)^k+(-s\zeta)^k\left(\frac{v_-v_+^*}\zeta\right)^k=(s\zeta)^k\frac{v_+v_-^*}\zeta+(-s\zeta)^k\frac{v_-v_+^*}\zeta$$
for all $s\in\R$ and $k\in\mathbb N$, the exponential computes to
\begin{equation}
\gG_{p}^{~t}(\Ppf)=e^{-\ln\left(\frac{t}{p}\right)v\wedge w}=\left(id-\frac{v_+v_-^*+v_-v_+^*}{\zeta}\right)+\left(\frac{t}{p}\right)^{-\zeta}~\frac{v_+v_-^*}{\zeta}+\left(\frac{t}{p}\right)^{\zeta}~\frac{v_-v_+^*}{\zeta}.
\label{eq:expo_mink_purepoleform}
\end{equation}
Therefore, 
$$\left(\frac{t}{p}\right)^{\zeta}\gG_{p}^{~t}(\Ppf)=\left(\frac{t}{p}\right)^{\zeta}\left(id-\frac{v_+v_-^*+v_-v_+^*}{\zeta}\right)+\frac{v_+v_-^*}{\zeta}+\left(\frac{t}{p}\right)^{2\zeta}~\frac{v_-v_+^*}{\zeta}.$$
Since $\zeta>0$, taking the limit $t\rightarrow 0$ yields \eqref{eq:lim_sim_min}. Since $|t/p|\leq 1$ for all $p\in(0,b)$ and $t\in (0,p]$, there is a constant $B\in\R$ such that \eqref{eq:bound_sim_min} holds.

\item For $\la v,w\ra$ spacelike, $\gG_{p}^{~t}(\Ppf)$ takes values in a 1-parameter subgroup  $O(\R^2)\subset O^+(\Rmn)$ of Euclidean rotations. Since $O(\R^2)$ is compact, $\gG_{p}^{~t}(\Ppf)$ is bounded for all $p,t\in (0,b)$ such that \eqref{eq:bound_sim_spa} holds for some $B\in\R$. To see that $\gG_p(\ppf)$ does not have a limit at $0$, denote by $i\zeta$ a nonzero eigenvalue of $v\wedge w$. Under a Lie group isomorphism $O(\R^2)\rightarrow \R~mod~2\pi$, the primitive $\gG_{p}^{~t}(\Ppf)$ gets mapped to $\pm\ln\left(\frac{\id}{p}|\zeta|\right)~mod~2\pi$ which does not have a limit for any $p\in (0,b)$ as $t$ tends to $0$. Hence, also $\gG_{p}^{~t}(\Ppf)$ does not have a limit for any $p\in (0,b)$ as $t$ tends to $0$.

\item Now let $\ppf$ be degenerate with $v_0\in \la v,w\ra\cap\Li^{n+1}$ and $\ti w\in \la v,w\ra$ such that $v\wedge w=v_0\wedge \ti w$. Then $(v\wedge w)^2=-\|\ti w\|^2v_0v_0^*$ and $(v\wedge w)^k=0$ for all $k\geq 3$. With $\bd$ as in \eqref{eq:defi_bound_ell}, we thus get
\begin{equation}
\frac 1{\bd(p,t)}\gG_{p}^{~t}(\Ppf)=
\frac{id-\ln\left(\frac{t}{p}\right)~v\wedge w-\frac{\left(\ln\left(\frac{t}{p}\right)\right)^2\|\ti w\|^2}2v_0v_0^*}{1+\left(\ln\left(\frac{t}{p}\right)\right)^2}.
\label{eq:temp_sim_deg}
\end{equation}
Taking the limit $t\rightarrow 0$ yields \eqref{eq:lim_sim_deg}. The bound \eqref{eq:bound_sim_deg} follows from the boundedness of $\frac 1{1+x^2}$, $\frac{x}{1+x^2}$, $\frac{x^2}{1+x^2}$ for all $x\in\R$.
\end{enumerate}
\qed
\end{proof}
The 1-parameter families $(\gG_{p}^{~t}(\ppf))_{t\in(0,b)}$ have simple Euclidean interpretations: For Minkowski $\ppf$, stereographically project $S^n\bs\{ \la v_+\ra\}$ to $\R^n$. Then $(\gG_{p}^{~t}(\ppf))_{t\in(0,b)}$ acts as a family of similarities with centre $\la v_-\ra$ of $\R^n$ and a $t$-dependent scale factor that tends towards infinity as $t\rightarrow 0$. The limiting map $\la v_+v_-^*\ra$ maps all points to the point at infinity, $\la v_+\ra$, except the centre, $\la v_-\ra$, where $\la v_+v_-^*\ra$ is not defined.\\
If $\ppf$ is spacelike, $(\gG_{p}^{~t}(\ppf))_{t\in(0,b)}$ takes values in a 1-parameter subgroup  $O(\R^2)\subset O^+(\Rmn)$ of Euclidean rotations. As $t$ approaches $0$, the rotation speed increases towards infinity.\\
For degenerate $\ppf$ and $v_0\in \la v,w\ra\cap\Li^{n+1}$, stereographically project $S^n\bs \{\la v_0\ra\}$ to $\R^n$. Then $(\gG_{p}^{~t}(\ppf))_{t\in(0,b)}$ acts as a family of translations of $\R^n$ by a vector of constant direction and $t$-dependent length which tends towards infinity as $t\rightarrow 0$. The limiting map $\la v_0v_0^*\ra$ maps all points to the point at infinity, $\la v_0\ra$, except $\la v_0\ra$ itself, where $\la v_0v_0^*\ra$ is not defined.

These limits are not elements of $\Pro O(\Rmn)$. Instead, they lie on the closure of $\Pro O(\Rmn)$ in $\Pro End(\Rmn)$. In fact, it can be shown (cf \cite{fu18}) that the closure of $\Pro O(\Rmn)$ in $\Pro End(\Rmn)$ is compact and precisely the union of $\Pro O(\Rmn)$ and the set $\{\la vw^*\ra ~|~v,w\in\Li^{n+1}\}\subset \Pro End(\Rmn)$. In contrast, $O^+(\Rmn)$ is closed in $End(\Rmn)$, but not compact. The limits \eqref{eq:lim_sim_min} and \eqref{eq:lim_sim_deg} can thus not be computed in $O^+(\Rmn)$, which demonstrates the advantage of the projective model.

\subsection{Primitives of pole forms}\label{section:genoneform}
In this section, we define pole forms and relate the limiting behaviour of their primitives at $0$ to that of pure pole forms.
\begin{definition}
A $\pro\ort(\Rmn)$-valued 1-form $\pf$ on $(0,b)$ is a \textit{pole form} if $\pf-\ppf$ is bounded with respect to the Euclidean norm $|\cdot|$ for a pure pole form $\ppf$ on $(0,b)$. 
\end{definition}
For a pole form $\pf$, the pure pole form $\ppf$ for which $\pf-\ppf$ is bounded is clearly unique. Therefore, we may define
\begin{definition}
Let $\pf$ be a pole form on $(0,b)$ with $\pf-\ppf$ bounded for the pure pole form $\ppf$. Then $\pf$ is called \textit{Minkowski, spacelike} or \textit{degenerate} according to whether $\ppf$ has that property. For Minkowski $\pf$, denote by $\zeta$ the positive eigenvalue of $v\wedge w$. If $\pf$ is Minkowski with $\zeta<1$, spacelike or degenerate, we say that $\pf$ is \textit{of the first kind}. Otherwise it is \textit{of the second kind}.
\end{definition}
We remark that these properties of a pole form are invariant under those diffeomorphisms of $(0,b)$ to $(\ti a,\ti b)$ which extend to diffeomorphisms of $(-\gep,b)$ to $(\ti a-\ti\gep,\ti b)$ for some $\gep,\ti\gep>0$ (cf \cite{fu18}).

Different techniques are necessary to investigate the limiting behaviour of primitives of pole forms of the first and of the second kind. We start with those of the first kind. 

\begin{lemma}\label{lemma:integrability_of_gauge_trafo1form}
Let $\ppf$ be a pure pole form of the first kind on $(0,b)$ and let $\chi$ be a bounded, continuous map from $(0,b)$ to $\pro\ort(\Rmn)$. Then the family $\left(\gG_{p}(\ppf)\chi\gG^{~p}(\ppf)\right)_{p\in(0,b)}$ of maps from $(0,b)$ to $\pro\ort(\Rmn)$ satisfies
$$~~~~~~\forall p\in(0,b):~~~~|\gG_{p}(\ppf)\chi\gG^{~p}(\ppf)|< \cB_{p},$$
where $(\cB_{p})_{p\in(0,b)}$ is a family of integrable functions on $(0,b)$ that satisfies
\begin{equation}
\forall p\in(0,b):~~~~\ltz t~\cB_{p}(t)=0.
\label{eq:prod_bound_withint}
\end{equation}
Moreover, if $\ppf$ is degenerate or spacelike, then
\begin{equation}
\lim_{p\rightarrow 0}\int_{0}^p\cB_{p}(t)\dd t=0.
\label{eq:double_limit_intbound}
\end{equation}

\end{lemma}
\begin{proof}
Let $\Ppf$ the the orthogonal lift of $\ppf$. If $\ppf$ is degenerate, it follows from the submultiplicativity of $|\cdot|$, Prop \ref{prop:boundednonzerosimplesections} and the assumed boundedness of $\chi$ that there is a $C\in\R$ such that
$$\forall p,t\in(0,b):~~~~|\gG_{p}^{~t}(\ppf)\chi(t)\gG_t^{~p}(\ppf)|\leq |\gG_{p}^{~t}(\Ppf)|^2 |\chi(t)|\leq C~ \bd(p,t)^2,$$
where $\bd(p,t)$ is as in \eqref{eq:defi_bound_ell}. Now set $\cB_{p}(t)=C \bd(p,t)^2$. For all $i\in\mathbb N$, by l'H\^opital's rule $(\ln x)^ix$ converges to zero as $x\rightarrow 0$ such that \eqref{eq:prod_bound_withint} holds. Moreover, for any $i\in\mathbb N$, the integral of $(\ln x)^i$ is a linear combination of terms of the form $x(\ln x)^k$ with $0\leq k\leq i$. Again, since all terms $x(\ln x)^k$ converge to zero as $x\rightarrow 0$, it follows that for all $p\in(0,b)$ the function $\cB_{p}$ is integrable and that \eqref{eq:double_limit_intbound} holds.

If $\ppf$ is spacelike, again by Prop \ref{prop:boundednonzerosimplesections}, $\gG_{p}(\Ppf)$ is bounded on $(0,b)$. Therefore $\gG_{p}(\ppf)\chi\gG^{~p}(\ppf)$ is bounded on $(0,b)$. Thus, we can choose $\cB_{p}$ to be the same constant for all $p\in(0,b)$. Clearly, a constant $B_{p}$ is integrable and satisfies \eqref{eq:prod_bound_withint} and \eqref{eq:double_limit_intbound}.

Let finally $\ppf$ be Minkowski with positive eigenvalue $\zeta<1$. Denote by $v_\pm$ the eigenvectors of $v\wedge w$ with eigenvalues $\pm\zeta$. According to the decomposition 
$$\ort(\Rmn)=\gL^2\Rmn=\la v_+\wedge v_-\ra\oplus\big(\la v_+\ra \wedge \la v_+,v_-\ra^\perp\big)\oplus\big(\la v_-\ra \wedge \la v_+,v_-\ra^\perp\big)\oplus \gL^2 \la v_+,v_-\ra^\perp,$$
 write $\chi=\chi_{+-}+\chi_{+\perp}+\chi_{-\perp}+\chi_{\perp\perp}+\R~id$. Using \eqref{eq:expo_mink_purepoleform}, $\gG_{p}(\ppf)\chi\gG^{~p}(\ppf)$ evaluated at $t$ reads
$$\gG_{p}^{~t}(\ppf)\chi(t)\gG_t^{~p}(\ppf)=\chi_{+-}(t)+\left(\frac{t}{p}\right)^{-\zeta}\chi_{+\perp}(t)+\left(\frac{t}{p}\right)^{\zeta}\chi_{-\perp}(t)+\chi_{\perp\perp}(t)+\R~id.$$
Thus, there is a constant $C$ such that
\begin{equation}
\forall p,t\in(0,b):~~~~\left|\gG_{p}^{~t}(\ppf)\chi(t)\gG_t^{~p}(\ppf)\right|\leq C\left(\left(\frac{t}{p}\right)^{-\zeta}+\left(\frac{t}{p}\right)^{\zeta}\right)=:\cB_{p}(t).
\label{eq:defi_boundfunc_mink}
\end{equation}
By assumption $\zeta<1$. Therefore, for all $p\in(0,b)$, the function $\cB_{p}$ is integrable over $(0,b)$ and the product $t\,\cB_{p}(t)$ converges to zero as $t$ tends towards $b$.\qed
\end{proof}
We remark that the bounding functions $\cB_{p}$ that we defined in \eqref{eq:defi_boundfunc_mink} for the Minkowski case do not satisfy \eqref{eq:double_limit_intbound}.

This Lemma can now be used to prove that a primitive of a pole form of the first kind factorizes into a continuous map from $[0,b)$ to $\Pro O(\Rmn)$ and the primitive of a pure pole form, whose limiting behaviour at $0$ we know from Prop \ref{prop:boundednonzerosimplesections}.

\begin{corollary}\label{coro:integrable_gauge_trafo}
Let $\pf$ be a pole form of the first kind so that $\pf-\ppf$ is bounded for the pure pole form $\xi$. For all $p\in (0,b)$, define the gauge transformed 1-form
\begin{equation}
\Gtr \ppf {p} \pf:=\gG_{p}(\ppf) \, \pf \, \gG^{~p}(\ppf)-\ppf.
\label{eq:gaugetrafo_poleformfirstkind}
\end{equation}
Then the primitives $\gG_{p}(\Gtr \ppf {p} \pf)$ have limits in $\Pro O(\Rmn)$ at $0$ and
\begin{equation}
\gG_{p}(\pf)=\gG_{p}(\Gtr \ppf {p} \pf) \, \gG_{p}(\ppf).
\label{eq:gaugetrafo_for_int_xi}
\end{equation}
Moreover, if $\ppf$ is spacelike or degenerate, 
\begin{equation}
\lim_{p\rightarrow 0}\gG_{p}^{~0}(\Gtr \ppf {p} \pf)=id.
\label{eq:limit_of_gammaoftiomega}
\end{equation}
\end{corollary}
\begin{proof}
Since $\ppf$ takes values in a fixed 1-dimensional subalgebra, the orthogonal lift $\Gtr \Ppf p \Pf$ of $\Gtr \ppf p \pf$ can be written as
\begin{equation}
\Gtr \Ppf p \Pf=\gG_{p}(\Ppf)\,\Pf\,\gG^{~p}\,(\Ppf)-\Ppf=\gG_{p}(\Ppf)\,(\Pf-\Ppf)\,\gG^{~p}(\Ppf).
\label{eq:bounded_1form_gatro}
\end{equation}
By assumption, $\pf-\ppf$ and hence $\Pf-\Ppf$ is bounded. Thus, by Lemma \ref{lemma:integrability_of_gauge_trafo1form} and \eqref{eq:bounded_1form_gatro}, $\Gtr \Ppf {p} \Pf$ is bounded by an integrable function $\cB_p$. By \cite[Ch 1.8]{dol79}, the $O^+(\Rmn)$-valued primitive $\gG_{p}(\Gtr \Ppf {p} \Pf)$ is continuous on $[0,b)$ and, in particular, has a limit in $O^+(\Rmn)$ at $0$. Thus, by \eqref{eq:equality_of_prims}, $\gG_{p}(\Gtr \ppf {p} \pf)$ has a limit in $\Pro O(\Rmn)$ at $0$. The relation \eqref{eq:gaugetrafo_for_int_xi} follows from \eqref{eq:gauge_trafo_beha} because $\Gtr \ppf {p} \pf$ is a gauge transform of $\pf$ by the map $\gG^{~p}(\xi)$. The limit \eqref{eq:limit_of_gammaoftiomega} follows from
$$\lim_{p\rightarrow 0}\left|\gG_p^{~0}(\Gtr \Ppf p \Pf)-id\right|\leq \lim_{p\rightarrow 0}\left|e^{\int_{0}^p|\Gtr \Ppf {p} \Pf|}-1\right| \leq \lim_{p\rightarrow 0}\left|e^{\int_0^p \cB_p(t)\dd t}-1\right|=0,$$
where the first inequality can be derived from the integral equation \eqref{eq:integral_eq} and in the last step we used \eqref{eq:double_limit_intbound}. \qed
\end{proof}
From the factorization \eqref{eq:gaugetrafo_for_int_xi} and Prop \ref{prop:boundednonzerosimplesections}, we conclude that the primitives of spacelike pole forms do not have limits at the singularity, while those of degenerate and Minkowski pole forms of the first kind do have limits. The following Corollary deals with these limits for the degenerate case. Below, in Lemma \ref{lemma:maintool_bound} and Prop \ref{prop:maintool_sing}, we then derive an analogous result for Minkowski pole forms of the first and second kind. 

\begin{corollary}\label{cor:maintool_sing_firstkind}
Let $\pf$ be a degenerate pole form on $(0,b)$ with $\pf-\ppf$ bounded for the pure pole form $\ppf$. Choose $v_0\in\la v, w\ra\cap\Li^{n+1}$. Then there is a continuous map $k:[0,b)\rightarrow \Li^{n+1}$ which satisfies
\begin{equation}
\forall p,t\in (0,b):~~~~\la k(t)\ra=\gG_t^{~p}(\pf)\la k(p)\ra,~~~~~~~\la k(0)\ra=\la v_0\ra,~~~~
\label{eq:props_k_deg}
\end{equation}
and is such that
\begin{equation}
~~~~~~~~~\forall p\in (0,b):~~~~\ltz \gG_p^{~t}(\pf)=\la k(p)v_0^*\ra,~~~~~\ltz \gG_t^{~p}(\pf)=\la v_0 k(p)^*\ra.
\label{eq:maintool_sing_firstkind}
\end{equation}
\end{corollary}
\begin{proof}
For $a\in \Pro End(\Rmn)$ choose $A\in End(\Rmn)$ such that $a=\la A\ra$. The adjoint $A^*$ with respect to the Minkowski inner product yields a well defined element $a^*:=\la A^*\ra\in \Pro End(\Rmn)$, which is independent of the chosen $A\in a$. If $a\in \Pro O(\Rmn)$, then $a^*=a^{-1}$. Now let $a:(0,b)\rightarrow \Pro O(\Rmn)$ be continuous with limit in $\Pro End(\Rmn)$ at $0$. Then, since the map $a\mapsto a^*$ is continuous, we have
\begin{equation}
\ltz a(t)^{-1}=\ltz a(t)^*=(\ltz a(t))^{*}.
\label{eq:limit_of_adjoint}
\end{equation}
Since $\gG_p^{~t}(\pf)=\left(\gG_t^{~p}(\pf)\right)^{-1}$, one limit in \eqref{eq:maintool_sing_firstkind} follows from the other due to \eqref{eq:limit_of_adjoint}. Now use Cor \ref{coro:integrable_gauge_trafo} and Prop \ref{prop:boundednonzerosimplesections} to find that
$$\forall p\in(0,b):~~~~\ltz \gG_{p}^{~t}(\go)=\gG_{p}^{~0}(\Gtr \xi {p} \go) \, \la v_0v_0^*\ra.$$
Choose $k(p)$ to be a continuous lift of $\gG_p^{~0}(\Gtr \xi {p} \go)\,\la v_0\ra$. Due to \eqref{eq:limit_of_gammaoftiomega}, $k(0)\in \la v_0\ra$. The other relation in \eqref{eq:props_k_deg} follows from \eqref{eq:maintool_sing_firstkind} and $\gG_p^{~q}(\pf)=\gG_p^{~t}(\pf)\gG_t^{~q}(\pf)$.\qed
\end{proof}

Lemma \ref{lemma:integrability_of_gauge_trafo1form} and Cor \ref{coro:integrable_gauge_trafo} do not hold for pole forms of the second kind. In particular, the gauge transforms \eqref{eq:gaugetrafo_poleformfirstkind} are in general not integrable and although the factorization \eqref{eq:gaugetrafo_for_int_xi} also exists in that case, it seems to be of little use because we do not know whether the first factor, $\gG_{p}(\Gtr \ppf {p} \pf)$, has a limit at $0$. Thus, we pursue a different strategy. The following Lemma and Proposition are based on ideas from \cite{lev46}.
\begin{lemma}\label{lemma:maintool_bound}
Let $\Pf$ be the orthogonal lift of a Minkowski pole form $\pf$ and $\ppf$ the pure pole form such that $\pf-\ppf$ is bounded. Denote by $\zeta$ the positive eigenvalue of $v\wedge w$. Then there is a $B\in\R$ and a $\ti b\in (0,b)$ such that
	\begin{equation}
	\forall p\in (0,\ti b)~\forall t\in(0,p]:~~~~~\left|\left(\frac{t}{p}\right)^\zeta \gG_{p}^{~t}(\Pf)\right|=\left|\left(\frac{t}{p}\right)^\zeta \gG^{~p}_t(\Pf)\right|<2B.
	\label{eq:upbound_ex_mink_pex}
	\end{equation}
\end{lemma}
\begin{proof}
The equality follows from $|A^{-1}|=|A^*|=|A|$ for $A\in O(\Rmn)$. We now prove the inequality. Let $\Ppf$ and $\Pf$ be the orthogonal lifts of $\ppf$ and $\pf$. For $p,t\in(0,b)$, one verifies by differentiation\footnote{or by using the integral equation \eqref{eq:integral_eq} for the 1-form $\Gtr \xi p \Pf$ and the identity $\gG_p(\Gtr \xi p \Pf)=\gG_p(\Pf)\gG^{~p}(\Ppf)$} that $\gG_{p}^{~t}(\Pf)$ satisfies
\begin{equation}
\left(\frac{t}{p}\right)^\zeta\gG_{p}^{~t}(\Pf)=\int^{t}_{p}\left(\frac{\tau}{p}\right)^\zeta\gG_{p}^{~\tau}(\Pf)(\Pf(\tau)-\Ppf(\tau))\left(\frac{t}{\tau}\right)^\zeta\gG_\tau^{~t}(\Ppf)\dd\tau+\left(\frac{t}{p}\right)^\zeta\gG_{p}^{~t}(\Ppf).
\label{eq:equationfory2}
\end{equation}
By assumption, $\pf-\ppf$ and hence $\Pf-\Ppf$ is bounded by some $D\in\R$ w.r.t. $|\cdot|$. By \eqref{eq:bound_sim_min}, $\left(\frac{t}{p}\right)^\zeta\gG_{p}^{~t}(\Ppf)$ is bounded by a constant $B\in\R$ for all $t\in (0,p]$. Therefore,
\begin{equation}
\forall p\in (0,b)~\forall t\in(0,p]:~~~~~\left(\frac{t}{p}\right)^\zeta|\gG_{p}^{~t}(\Pf)|\leq DB\int_{p}^{t} \left(\frac{\tau}{p}\right)^\zeta|\gG_{p}^{~\tau}(\Pf)|\dd\tau+B.
\label{eq:normyt}
\end{equation}
We now prove by contradiction that there is a $\ti b\in(0,b)$ such that
\begin{equation}
\forall p\in(0,\ti b)~\forall t\in (0,p]:~~~\left|\left(\frac{t}{p}\right)^\zeta\gG_{p}^{~t}(\Pf)\right|\neq 2B.
\label{eq:notequal_bound}
\end{equation}
Namely, choose $\ti b\in(0,b)$ sufficiently small such that 
\begin{equation}
DB\int_{0}^{\ti b} \dd\tau<\frac 12.
\label{eq:choiceoft0}
\end{equation}
Suppose there existed a $p\in (0,\ti b)$ and a $t\in(0,p]$ such that $\left|\left(\frac{t}{p}\right)^\zeta\gG_{p}^{~t}(\Pf)\right|=2B$. Denote by $\bar t$ the smallest of those $t$. Then evaluating \eqref{eq:normyt} at $\bar t$ and using \eqref{eq:choiceoft0} yields $2B<\frac 12 2B+B$, a contradiction. Hence, indeed \eqref{eq:notequal_bound} holds. 

To conclude \eqref{eq:upbound_ex_mink_pex}, we note that for all $p\in (0,\ti b)$, the map $t\mapsto \left|\left(\frac{t}{p}\right)^\zeta\gG_{p}^{~t}(\Pf)\right|$ is continuous, has the value $|id|$ at $t=p$ and, by the above, nowhere on $(0,p]$ attains the value $2B\geq 2|id|>|id|$. Thus necessarily \eqref{eq:upbound_ex_mink_pex} holds. \qed
\end{proof}

We can now prove the analogue of Cor \ref{cor:maintool_sing_firstkind} for Minkowski pole forms.
%%%     P R O P    P R O P    P R O P   P R O P   P R O P   P R O P   P R O P %%%
\begin{proposition}\label{prop:maintool_sing}
Let $\Pf$ be the orthogonal lift of a Minkowski pole form $\pf$ on $(0,b)$ with $\pf-\ppf$ bounded for the pure pole form $\ppf$. Denote by $v_\pm$ eigenvectors of $v\wedge w$ with eigenvalues $\pm\zeta$, $\zeta>0$. Then there is a continuous map $k:[0,b)\rightarrow \Li^{n+1}$ which satisfies 
	\begin{equation}
	\forall p\in (0,b):~~~~k(t)=\left(\frac{t}{p}\right)^{-\zeta}\gG^{~p}_t(\Pf)k(p),~~~~~~~k(0)=\frac{v_+}{\ipl v_+,v_-\ipr}~~~~~~~~~~~
	\label{eq:props_of_k}
	\end{equation}
	and is such that
	\begin{equation}
	\forall p\in(0,b):~~~~\lim_{t\rightarrow 0}\left(\frac{t}{p}\right)^{\zeta}\gG^{~t}_{p}(\Pf)= k(p)v_-^*,~~~~~\lim_{t\rightarrow 0}\left(\frac{t}{p}\right)^{\zeta}\gG^{~p}_{t}(\Pf)= v_-k(p)^*.\label{eq:mainlimit_tool_md}
	\end{equation}
	In particular,
	\begin{equation}
	\forall p\in(0,b):~~~~\lim_{t\rightarrow 0}\gG^{~t}_{p}(\pf)=\la k(p)v_-^*\ra,~~~~~~~~~~~\lim_{t\rightarrow 0}\gG^{~p}_{t}(\pf)=\la v_-k(p)^*\ra.~~~~
	\label{eq:mainlimit_tool_proj}
	\end{equation}
\end{proposition}
\begin{proof}
First, we note that \eqref{eq:mainlimit_tool_proj} follows from \eqref{eq:mainlimit_tool_md} and one limit in \eqref{eq:mainlimit_tool_md} follows from the other because $(\gG_p^{~t}(\Pf))^*=\gG_t^{~p}(\Pf)$ and taking the adjoint is continuous. We prove the first limit in \eqref{eq:mainlimit_tool_md}.

Let $\Pf$ and $\Ppf$ be the orthogonal lifts of $\pf$ and $\ppf$, respectively. Consider again the identity \eqref{eq:equationfory2}. For $t\in (0,p)$ it can be written in the form
\begin{equation}
\left(\frac{t}{p}\right)^{\zeta}\gG_{p}^{~t}(\Pf)=\int_p^{0}M(t,\tau,p)\dd\tau+\left(\frac{t}{p}\right)^{\zeta}\gG_{p}^{~t}(\Ppf),
\label{eq:equationfory}
\end{equation}
where
$$M(t,\tau,p):=\Theta(\tau-t)\left(\frac{\tau}{p}\right)^{\zeta}\gG_{p}^{~\tau}(\Pf)(\Pf(\tau)-\Ppf(\tau))\left(\frac{t}{\tau}\right)^{\zeta}\gG_\tau^{~t}(\Ppf)~~~\text{with}~~~\Theta(x)=\left\{\begin{array}{cl}
1&\text{for}~x\geq 0\\
0&\text{for}~x<0
\end{array}\right.$$
In order to prove \eqref{eq:mainlimit_tool_md}, we wish to take the limit $t\rightarrow 0$ of \eqref{eq:equationfory}. To that end, we convince ourselves, that we may evaluate the limit $t\rightarrow 0$ under the integral in \eqref{eq:equationfory} by applying the dominated convergence theorem. Let first $\ti b$ be as in Lemma \ref{lemma:maintool_bound} and fix $p\in(0,\ti b)$. Let $(t_i)_{i\in\mathbb N}$ be any sequence in $(0,p]$ converging to $0$ and consider the sequence $\left(\tau\mapsto M(t_i,\tau,p)\right)_{i\in\mathbb N}$ of functions from $(0,p]$ to $\R^{(n+2)^2}$. Due to \eqref{eq:lim_sim_min}, this sequence of functions converges pointwise. Moreover, due to Lemma \ref{lemma:maintool_bound}, the assumed boundedness of $\Pf-\Ppf$, and \eqref{eq:bound_sim_min}, the sequence of functions is uniformly bounded by some constant. Thus, the conditions of the dominated convergence theorem are satisfied and we conclude
\begin{align}
\ltz \left(\frac{t}{p}\right)^{\zeta}\gG_{p}^{~t}(\Pf)=&\int_p^{0}\ltz M(t,\tau,p)\dd\tau+\ltz\left(\frac{t}{p}\right)^{\zeta}\gG_{p}^{~t}(\Ppf)\notag\\
=&\left[\int_{p}^0 \left(\frac{\tau}{p}\right)^{\zeta}\gG_{p}^{~\tau}(\Pf)(\Pf(\tau)-\Ppf(\tau))\dd\tau+id\right]\frac{v_+}{\ipl v_+,v_-\ipr}v_-^*=:k(p)v_-^*\label{eq:defi_k},
\end{align}
where we used \eqref{eq:lim_sim_min}. Again due to Lemma \ref{lemma:maintool_bound} and the assumed boundedness of $\Pf-\Ppf$, the integrand in $k$ is bounded uniformly in $p$ and so the integral in $k$ exists for all $p\in(0,\ti b)$ and $\lim_{p\rightarrow 0}k(p)=v_+/\ipl v_+,v_-\ipr$. The first identity in \eqref{eq:props_of_k} also follows from \eqref{eq:defi_k} and \eqref{eq:propsgG2}.

The map $k$ has to take values in the light cone because for all $p\in (0,\ti b)$ we have
$$v_-\|k(p)\|^2v_-^*=\ltz \left(\frac{t}{p}\right)^{2\zeta}\gG_t^{~p}(\Pf)\gG^{~t}_{p}(\Pf)=\ltz \left(\frac{t}{p}\right)^{2\zeta}id=0.$$
This proves the Proposition for all $p\in (0,\ti b)$. But using \eqref{eq:propsgG2} it can easily be seen to hold for all $p\in (0,b)$.\qed

\end{proof}

%-------------------------------------------------------------------------------
%-------------------------------------------------------------------------------
\section{Pole of first order}\label{section:first_order_pole}
%-------------------------------------------------------------------------------
%-------------------------------------------------------------------------------
In this section we investigate the limiting behaviour at $0$ of the Darboux and Calapso transforms of a polarized curve $(\la c\ra,Q)$ on $(0,b)$, where $\la c\ra$ has a smooth extension to $(-\gep,b)$ for some $\gep>0$ and $Q$ has a pole of first order at $0$. For convenience, we further assume that $\la c\ra$ and $Q$ have smooth extensions to some $(0,b+\ti\gep)$, $\ti\gep>0$.

We assume that $Q_t=Q(t)\dd t^2$ satisfies $Q(t)<0$. This is convenient and no restriction of our results because replacing $Q$ by $-Q$ has the same effect as replacing $\gl$ by $-\gl$ and there is no restriction on the range or sign of $\gl$. With this assumption, we can choose a smooth lift $c$ of $\la c\ra$ such that $\Qend(t)=\frac{Q(t)}{\|\pr c(t)\|^2}=-\frac{1}{t}$ and the 1-form associated to $(\la c\ra,Q)$ reads
$$\of_t=-\frac{\dd t}{t}~c(t)\wedge \pr c(t)+\R~id.$$
Now define the degenerate pure pole form
$$\ppf_t:=-\frac{\dd t}{t}c(0)\wedge \pr c(0)+\R ~id.$$
Then $\gl\of-\gl\ppf$ has a finite limit at $0$ because
$$|\of(t)-\ppf(t)|=\left|\frac{c(t)\wedge \pr c(t)-c(0)\wedge \pr c(0)}{t}\right|=\left|c(t)\wedge \frac{\pr c(t)-\pr c(0)}{t}+\frac{c(t)-c(0)}{t}\wedge \pr c(0)\right|$$
and $\la c\ra$ is assumed to be smoothly extendible to $(-\gep,b)$. Thus, $\gl\of-\gl\ppf$ is bounded on $(0,b)$ and $\gl\of$ is a degenerate pole form. In particular, it is a pole form of the first kind.

With the help of Cor \ref{coro:integrable_gauge_trafo} we can now show that every Calapso transform of $(\la c\ra,Q)$ converges.
\begin{theorem}\label{theorem:conv_cal_trafo}
Let $(\la c\ra,Q)$ be a polarized curve on $(0,b)$ such that $\la c\ra$ has a regular extension to $(-\gep,b)$ and $Q$ has a pole of first order at $0$. Then for all $p\in (0,b)$ and $\gl\in\R$, the $\gl$-Calapso transform $\la c_{\gl,p}\ra$ normalized at $p$ has a limit at $0$.
\end{theorem}
\begin{proof}
If $\gl=0$, then $\la c_{\gl,p}\ra=\la c\ra$ and the statement is trivial. So let $\gl\neq 0$. Factorize $\gG_p(\gl\of)$ as in \eqref{eq:gaugetrafo_for_int_xi} to get
$$\la c_{\gl,p}\ra=\gG_p(\gl\of) \, \la c\ra=\gG_p(\Gtr {\gl\xi} p {\gl\of}) \, \gG_p({\gl\xi}) \, \la c\ra.$$
Since $\gl\of$ is of the first kind, from Cor \ref{coro:integrable_gauge_trafo} we know that $\gG_p(\Gtr {\gl\xi} p {\gl\of})$ has a limit in $\Pro O(\Rmn)$ at $0$. To see that $\gG_p({\gl\xi}) \, \la c\ra$ has a limit at $0$, write
\begin{equation}
\gG_p^{~t}({\gl\xi})=\La id-\ln\left(\frac{t}{p}\right)\gl ~c(0)\wedge \pr c(0)-\frac{\left(\ln\left(\frac{t}{p}\right)\right)^2\gl^2\|\pr c(0)\|^2}{2}c(0)c(0)^*\Ra .
\label{eq:expr_primxi_expl}
\end{equation}
Since both $\frac{1}t \ipl c(0),c(t)\ipr$ and $\frac{1}t \ipl\pr c(0),c(t)\ipr$ have finite limits at $0$, and 
$$\ltz t \ln\left(\frac{t}{p}\right)=0,~~~~~\ltz t\left(\ln\left(\frac{t}{p}\right)\right)^2=0,$$
the product of \eqref{eq:expr_primxi_expl} with $\la c(t)\ra$ converges to $\la c(0)\ra$. Thus, the limit of $\la c_{\gl,p}\ra$ at $0$ exists and is equal to $\gG_p^{~0}(\Gtr {\gl\xi} p {\gl\of}) \, \la c(0)\ra$. \qed
\end{proof}

The Darboux transforms of $(\la c\ra,Q)$ converge to $\la c(0)\ra$.
\begin{theorem}\label{theorem:curve_first_ord_pole}
Let $(\la c\ra,Q)$ be a polarized curve on $(0,b)$ such that $\la c\ra$ has a regular extension to $(-\gep,b)$ and $Q$ has a pole of first order at $0$. Then for all $\gl\in\R\bs\{0\}$, the limit of any $\gl$-Darboux transform $\la \ch\ra$ of $(\la c\ra,Q)$ at $0$ is $\la c(0)\ra$, 
$$\lim_{t\rightarrow 0}\la \ch(t)\ra=\la c(0)\ra.$$
\end{theorem}
\begin{proof}
From Cor \ref{cor:maintool_sing_firstkind} we know that for every $\gl\neq 0$ there is a continuous map $k_\gl:[0,b)\rightarrow \Li^{n+1}$ such that
$$\ltz \gG_t^{~p}(\gl\of)=\la c(0)k_\gl(p)^*\ra,$$
and $k_\gl$ satisfies \eqref{eq:props_k_deg}. Therefore, if $\la \ch_p\ra\neq \la k_\gl(p)\ra$, then
$$\ltz \la \ch(t)\ra=\ltz \gG_t^{~p}(\gl\of) \, \la \ch_p\ra=\la c(0)k_\gl(p)^*\ra\la \ch_p\ra=\la c(0)\ra.$$
If on the other hand $\la \ch_p\ra = \la k_\gl(p)\ra$, then we use the properties \eqref{eq:props_k_deg} of $k_\gl$ to find 
$$\ltz \la \ch(t)\ra=\ltz \gG_t^{~p}(\gl\of)  \, \la k_\gl(p)\ra=\ltz \la k_\gl(t)\ra=\la c(0)\ra.$$
\qed
\end{proof}

\begin{figure}
\centering
\includegraphics[width=0.75\textwidth]{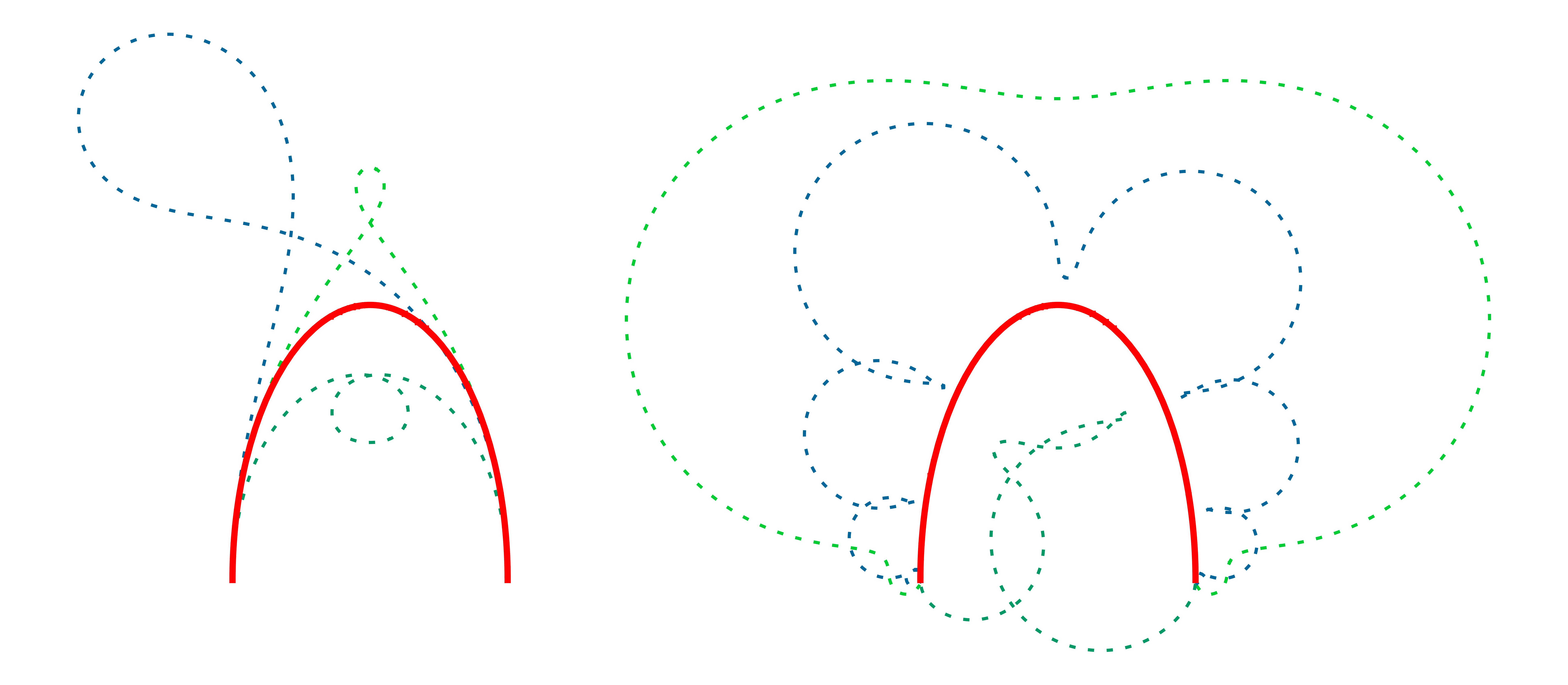}

\caption{Darboux transforms (dashed) of a half ellipse (solid) with respect to a polarization with a pole of first order at each end. The spectral parameter $\gl$ is positive (negative) on the right (left) side.}
\end{figure}

% ----------------------------------------------------------------------------------------------------
% -------------------------- POLE OF SECOND ORDER ----------------------------------------------------
% ----------------------------------------------------------------------------------------------------

\section{Pole of second order}\label{section:sec_order_pole}
We now come to the more intricate and diverse case of a polarized curve $(\la c\ra,Q)$ on $(0,b)$, where $\la c\ra$ is smoothly extendible to $(-\gep,b)$ for some $\gep>0$ and $Q$ has a pole of second order at $0$. In this case, also $\gl\of$ has a pole of second order at $0$ and we cannot apply the results of Sect \ref{section:maintool} directly. Instead, we use a gauge transformation 
$$\gl\of\mapsto \gatr \hyfr {\gl\of}=\la g\ra^{-1}\gl\of\la g\ra+\la g\ra^{-1}\dd\la g\ra$$
as in \eqref{eq:gauge_trafo_beha} to write
\begin{equation}
\gG_{p}(\gl\of)=\la g(p)\ra \, \gG_{p} (\gatr \hyfr {\gl\of}) \, \la g\ra^{-1},~~~~~~~~\gG^{~p}(\gl\of)=\la g\ra \, \gG^{~p}(\gatr \hyfr {\gl\of}) \, \la g(p)\ra^{-1}.
\label{eq:gauge_trafo_secordpole}
\end{equation}
As we will see in Sect \ref{section:sing_gauge}, it is possible to choose $\la g\ra$ such that $\gatr \hyfr {\gl\of}$ is a pole form and, in particular, has only a pole of first order. This is achieved by a gauge transformation which is singular in the sense that the transforming map $\la g\ra:(0,b)\rightarrow \Pro O(\R^{n+2}_1)$ converges to a map of the kind $\la vw^*\ra\in\Pro End(\Rmn)$ with $v,w\in\Li^{n+1}$ as $0$ is approached. In Sects \ref{section:negla} and \ref{section:posla}, we will then apply the results of Sect \ref{section:maintool} to the pole form $\gatr \hyfr {\gl\of}$ to investigate the behaviour of the Darboux and Calapso transforms of $(\la c\ra,Q)$ at the singular point $0$. 

Again, for convenience, we assume that $\la c\ra$ and $Q$ have smooth extensions to some $(0,b+\ti\gep)$, $\ti\gep>0$.

\subsection{The singular gauge transformation}\label{section:sing_gauge}
We assume that the polarization $Q$ is of the form $Q_t=Q(t)\dd t^2=\frac{\dd t^2}{\id^2}$. Again, this is no restriction of our results: we can assume $Q(t)$ to be positive because replacing $Q$ by $-Q$ has the same effect as replacing $\gl$ by $-\gl$ and we make no restriction on the range of $\gl$. Furthermore, the Darboux and Calapso transformations are invariant under reparametrizations of $\la c\ra$. Since $Q$ has a pole of second order (see Def \ref{defi:pole}), there certainly is a parameter $t$ for $\la c\ra$ such that $Q_t=\frac{\dd t^2}{\id^2}$ and $\la c\ra$ remains smoothly extendible to $(-\gep,b)$.

For the gauge transformation we use a product $g=FR$ of a frame $F:(0,b)\rightarrow O(R^{n+2}_1)$, smoothly extendible to $(-\gep,b)$, and a particularly simple singular factor $R:(0,b)\rightarrow O(R^{n+2}_1)$.

We first construct the frame $F$. To this end, let $c$ be the flat lift of $\la c\ra$, for which $\|\pr c\|^2=1$. Then, with the above assumptions, $\Qend(t)=\frac{1}{t^2}$. Let $N_1,...,N_{n-1}$ be parallel, orthonormal unit normal fields of $c$ that satisfy\footnote{This condition means that $\la N_1,...,N_{n-1}\ra^\perp\cap \Li^{n+1}$ is the curvature circle congruence of $\la c\ra$.} $\ipl \dd N_i,\pr c\ipr=0$ for $i=1,...,n-1$ and let $\bar c:(0,b)\rightarrow \Li^{n+1}$ be the unique map such that $\ipl c,\bar c\ipr=-1$ and $\la c,\bar c\ra=\la \pr c,N_1,...,N_{n-1}\ra^\perp$. Now define the frame $F:(0,b)\rightarrow O(\Rmn)$ to map $t$ to the orthogonal transformation $F(t)$ which maps a $t$-independent pseudo-orthonormal basis $\{\baszer,\basinf, \bastan,\basnor_{1},...,\basnor_{n-1}\}$ of $\Rmn$ to the $t$-dependent basis $\{c(t),\bar c(t),\pr c(t),N_1(t),...,N_{n-1}(t)\}$. The Maurer-Cartan form of $F$ is then of the form
$$F^{-1}\dd F=-\basinf\wedge \bastan~\dd t+\baszer\wedge \Fr^\perp$$
for some $\la \baszer,\basinf\ra^\perp$-valued 1-form $\Fr^\perp$. 

Next, define the singular factor $R:(0,b)\rightarrow O(\R^{n+2}_1)$ by
$$R(t)\baszer=\frac 1 t \,\baszer,~~~~R(t)\basinf=t\,\basinf,~~~~R\big|_{\la \baszer,\basinf\ra^\perp}=id.$$
Then $R$ has Maurer-Cartan form
$$(R^{-1}\dd R)_t=-\frac{\dd t}{\id}\baszer\wedge\basinf.$$
The product $\hyfr:=F R$ is a singular frame for the singular lifts $t^{-1}c(t)$ and $t\bar c(t)$ of $\la c(t)\ra$ and $\la \bar c(t)\ra$, respectively. It satisfies
\begin{equation}
\lim_{t\rightarrow 0}\la \hyfr(t)\ra=\lim_{t\rightarrow 0}\la F(t) \, R(t)\ra=\lim_{t\rightarrow 0}\la F(t) \, t  \, R(t)\ra=\la c(0)\basinf^*\ra\in \Pro End(\Rmn).
\label{eq:limit_g_H}
\end{equation}
Using $\la \hyfr\ra$ for a gauge transformation \eqref{eq:gauge_trafo_beha} of $\gl\of$, the gauge-transformed 1-form $\gatr \hyfr {\gl\of}$ reads
\begin{align}
\gatr \hyfr {\gl\of}=&R^{-1}F^{-1}\gl\of FR+R^{-1}\dd R+R^{-1}(F^{-1}\dd F) R,\notag\\
(\gatr \hyfr {\gl\of})_t=&-\frac{\dd t}{\id}\big(-\left(\gl\baszer-\basinf\right)\wedge \bastan+\baszer\wedge\basinf\big)+t~\baszer\wedge\Fr^\perp_t+\R~id.\label{eq:gauge_trafo_eta}
\end{align}
Indeed, $\gatr \hyfr {\gl\of}$ has only a pole of first order and the pure pole form $\ppf_\gl$ given by
\begin{align*}
\ppf_\gl(t)\dd t:=&-\frac{\dd t}{\id}\big(-\left(\gl \baszer-\basinf\right)\wedge \bastan+\baszer\wedge\basinf\big)+\R~id=\\
=&-\frac{\dd t}{\id}\left(\baszer- \bastan \right)\wedge\left(\basinf-\gl \bastan \right)+\R~id
\end{align*}
is such that $|\gatr \hyfr {\gl\of}-\ppf_\gl|$ is bounded\footnote{In fact, even $\frac 1 t |(\gatr \hyfr {\gl\of}-\ppf_\gl)_t|$ is bounded on $(0,b)$, which will be useful in the proof of Thm \ref{theorem:darb_tr_sec_ord_pos}.} on $(0,b)$. Thus, $\gatr \hyfr {\gl\of}$ is a pole form.

The next Lemma, which summarizes the algebraic properties of $\ppf_\gl$, can be verified by direct computation.
\begin{lemma}\label{lemma:xi_la_data}
%$\ppf_\gl$ takes values in the subalgebra $\La \left(\baszer-\bastan \right)\wedge\left(\basinf-\gl \bastan \right)+\R~id\Ra$. 
Define
\begin{equation}
v_\pm:=\sqrt{1-2\gl}(\gl\baszer-\basinf)\pm (\gl\baszer+\basinf-2\gl\bastan),
\label{eq:eigenvectors}
\end{equation}
such that $\ppf_\gl$ can be written as
$$\ppf_\gl(t)\dd t=-\frac{\dd t}{t}~\sqrt{1-2\gl} \frac{v_-\wedge v_+}{\ipl v_-,v_+\ipr}+\R~id=:-\frac{\dd t}{t}~v\wedge w+\R~id.$$
If $1-2\gl>0$, then $\la v,w\ra$ is Minkowski and $v_\pm$ are eigenvectors of $v\wedge w$ with real eigenvalues $\pm \sqrt{1-2\gl}$. If $1-2\gl=0$, then $\la v,w\ra$ is degenerate and $v_+=-v_-\in\la v,w\ra$ is null. If $1-2\gl<0$, then $\la v,w\ra$ is spacelike and $v_\pm$ are complex conjugate eigenvectors of $v\wedge w$ with imaginary eigenvalues $\pm \sqrt{1-2\gl}\in i\R$.

In particular, for $\gl>0$, the pole form $\gl\of$ is of the first kind. For $\gl\leq 0$, it is of the second kind.
\end{lemma}
By Prop \ref{prop:maintool_sing}, Cor \ref{coro:integrable_gauge_trafo} and Prop \ref{prop:boundednonzerosimplesections}, for $1-2\gl \geq 0$, the primitives $\gG^{~t}_{p}(\gatr \hyfr {\gl\of})$ converge as $t\rightarrow 0$ while for $1-2\gl<0$ they do not have a limit at $0$. We treat these cases separately in the next two sections.

%-------------------------------------------------------------------------------------
%-------------------------------------------------------------------------------------
\subsection{The behaviour at the singularity for $1-2\gl\geq 0$}\label{section:negla}
To determine the behaviour of the Darboux and Calapso transforms of $\la c\ra$ at $0$, we seek the limits $t\rightarrow 0$ of 
\begin{align}
\la \ch(t)\ra&=\la g(t)\ra \, \gG_t^{~p}(\gatr g {\gl\of}) \, \la g(p)\ra^{-1}\la \ch_p\ra,\label{eq:expr_darbouxtrafo_secord}\\
\la c_{\gl,p}(t)\ra&=\la g(p)\ra \, \gG_p^{~t}(\gatr g {\gl\of}) \, \la g(t)\ra^{-1}\la c(t)\ra=\la g(p)\ra \, \gG_p^{~t}(\gatr g {\gl\of}) \, \la \baszer\ra.\label{eq:expr_caltrafo_secord}
\end{align}
By \eqref{eq:limit_g_H}, $\la g(t)\ra$ has a limit at $0$. Since $1-2\gl\geq 0$, by Lemma \ref{lemma:xi_la_data} the pole form $\gatr g {\gl\of}$ is Minkowski or degenerate such that its primitives also have limits at $0$. Very little work is required to prove convergence of the Darboux and Calapso transforms at $0$ in the following two theorems using the results of Sect \ref{section:maintool}.

\begin{theorem}\label{theorem:darb_tr_sec_ord_pole}
Let $(\la c\ra,Q)$ be a polarized curve on $(0,b)$ where $\la c\ra$ is regularly extendible to $(-\gep,b)$ and $Q$ has a pole of second order at $0$. Then for every $\gl\in\R\bs\{0\}$ with $1-2\gl\geq 0$, all $\gl$-Darboux transforms $\la \ch\ra$ of $\la c\ra$ converge to $\la c(0)\ra$ as $0$ is approached, 
$$\lim_{t\rightarrow 0}\la \ch(t)\ra=\la c(0)\ra.$$
\end{theorem}
\begin{proof}
Let first $1-2\gl>0$. To find the behaviour of the Darboux transforms $\la \ch\ra$ of $\la c\ra$ at $0$, we want to take the limit $t\rightarrow 0$ of \eqref{eq:expr_darbouxtrafo_secord}. According to \eqref{eq:limit_g_H} and Prop \ref{prop:maintool_sing},
$$\lim_{t\rightarrow 0}\la \hyfr(t)\ra=\la c(0)\basinf^*\ra,~~~~\lim_{t\rightarrow 0}\gG^{~p}_t(\gatr \hyfr {\gl\of})=\la v_- k_\gl(p)^*\ra$$
for a continuous map $k_\gl:[0,b)\rightarrow \Li^{n+1}$ which satisfies \eqref{eq:props_k_deg} and $v_\pm$ given by \eqref{eq:eigenvectors}.

If $\la \ch_p\ra\neq \la \hyfr(p)\ra\la k_\gl(p)\ra$. Then
$$\ltz \la \ch(t)\ra=\ltz \la g(t)\ra \, \gG^{~p}_t(\gatr \hyfr {\gl\of}) \, \la g(p)\ra^{-1}\la \ch_p\ra=\la c(0)\basinf^*\ra\la v_- k_\gl(p)^*\ra \la g(p)\ra^{-1}\la \ch_p\ra=\la c(0)\ra,$$
where we used $\basinf^*(v_-)\neq 0$. 

If on the other hand $\la \ch_p\ra = \la \hyfr(p)\ra\la k_\gl(p)\ra$, then we use \eqref{eq:props_of_k} to get
$$\ltz \la \ch(t)\ra=\ltz \la \hyfr(t)\ra \, \gG_t^{~p}(\gatr \hyfr {\gl\of}) \, \la k_\gl(p)\ra=\ltz \la \hyfr(t)\ra \la k_\gl(t)\ra=\la c(0)\basinf^*\ra \la v_+^*\ra=\la c(0)\ra,$$
where in the last equality we used $\basinf^*(v_+)\neq 0$. 

This proves the theorem for $1-2\gl>0$. For $1-2\gl=0$, the proof works completely analogous with Cor \ref{cor:maintool_sing_firstkind} in place of Prop \ref{prop:maintool_sing} and \eqref{eq:props_k_deg} instead of \eqref{eq:props_of_k}.\qed
\end{proof}

For the Calapso transforms, even less work is necessary.
\begin{theorem}\label{prop:limit_cal_tr_sec_ord_neg}
Let $(\la c\ra,Q)$ be a polarized curve on $(0,b)$ where $\la c\ra$ has a regular extension to $(-\gep,b)$ and $Q$ has a pole of second order at $0$. Then all $\gl$-Calapso transforms with $1-2\gl\geq 0$ converge as $t\rightarrow 0$.
\end{theorem}
\begin{proof}
In \eqref{eq:expr_caltrafo_secord}, we use Prop \ref{prop:maintool_sing} for $1-2\gl>0$, Cor \ref{cor:maintool_sing_firstkind} for $1-2\gl=0$ and $v_-^*(\baszer)\neq 0$ to get
$$\ltz \la c_{\gl,p}(t)\ra=\ltz\la \hyfr(p)\ra  \, \gG_p^{~t}(\gatr \hyfr {\gl\of}) \, \la \baszer\ra=\la \hyfr(p)\ra \la k_\gl(p)v_-^*\ra\la \baszer\ra =\la \hyfr(p)\ra\la k_\gl(p)\ra.$$
\qed
\end{proof}

%-------------------------------------------------------------------------------------
%-------------------------------------------------------------------------------------
\subsection{The behaviour at the singularity for $1-2\gl<0$}\label{section:posla}
In this case, $\ppf_\gl$ takes values in an algebra of infinitesimal Euclidean rotations. Thus, by Prop \ref{prop:boundednonzerosimplesections}, the primitives of $\ppf_\gl$ do not have a limit at $0$ and hence by Cor \ref{cor:maintool_sing_firstkind} neither do the primitives of $\gatr \hyfr{\gl\of}$ converge. From this it follows that the $\gl$-Calapso transforms of $\la c\ra$ do not converge.

\begin{theorem}\label{theorem:cal_tra_at_b_secord_pos}
Let $(\la c\ra,Q)$ be a polarized curve on $(0,b)$ such that $\la c\ra$ has a regular extension to $(-\gep,b)$ and $Q$ has a pole of second order at $0$. For $\gl\in\R\bs\{0\}$ and $p\in(0,b)$, the $\gl$-Calapso transform normalized at $p$ tends\footnote{$\la c_{\gl,p}\ra$ tends towards $\la C_{\gl,p}\ra$ in the sense that the distance between $\la c_{\gl,p}(t)\ra$ and $\la C_{\gl,p}(t)\ra$ measured with respect to an arbitrary representative metric of the conformal structure of $S^n$ tends to zero as $t\rightarrow 0$.} towards the circular motion
$$\la C_{\gl,p}(t)\ra:=\la \hyfr(p)\ra \, \gG^{~0}_{p}(\Gtr {\ppf_\gl}{p}{\gatr\hyfr{\gl\of}}) \, \gG_{p}^{~t}(\ppf_\gl)\La \baszer \Ra $$
whose speed tends to infinity as $t\rightarrow 0$. In particular, it does not have a limit point at $0$, but the limit circle
\begin{equation}
\LiSe[\la c_{\gl,p}\ra]:=\hyfr(p) \, \gG^{~0}_{p}(\Gtr {\xi_\gl}{p}{\gatr\hyfr{\gl\of}}) \, \la \baszer,\basinf,\bastan\ra\cap\Li^{n+1}.
\label{eq:limit_circle_cal}
\end{equation}
\end{theorem}
\begin{proof}
Let $\gO$ and $\Ppf_\gl$ be the orthogonal lifts of $\go$ and $\ppf_\gl$, respectively. Since $\gG_{p}(\Gtr {\Ppf_\gl}{p}{\gatr\hyfr{\gl\Of}})$ is continuous on $[0,b)$, the inner product of the bounded lifts
\begin{align*}
c_{\gl,p}(t)=&g(p) \, \gG_{p}^{~t}(\Gtr {\Ppf_\gl}{p}{\gatr\hyfr{\gl\Of}}) \, \gG_{p}^{~t}(\Ppf_\gl) \, \baszer,\\
C_{\gl,p}(t)=&g(p) \, \gG_{p}^{~0}(\Gtr {\Ppf_\gl}{p}{\gatr\hyfr{\gl\Of}}) \, \gG_{p}^{~t}(\Ppf_\gl) \, \baszer
\end{align*}
of $\la c_{\gl,p}\ra$ and $\la C_{\gl,p}\ra$, respectively, tends to zero. Since these lifts do not converge to zero at $0$, we conclude that $\la c_{\gl,p}\ra$ approaches $\la C_{\gl,p}\ra$.

That $\la C_{\gl,p}\ra$ is a parametrization of the circle $\hyfr(p) \, \gG_{p}^{~0}(\Gtr {\xi_\gl}{p}{\gatr\hyfr{\gl\of}}) \, \la \baszer,\basinf,\bastan\ra$ follows because $\Ppf_\gl$ takes values in the subalgebra $\la (\baszer-\bastan)\wedge (\basinf-\gl\bastan)\ra$: the 1-parameter group $\gG_{p}(\ppf_\gl)$ of Euclidean rotations moves the point $\la \baszer\ra$ on the circle $\la \baszer,\basinf, \bastan\ra\cap\Li^{n+1}$ with speed increasing towards infinity as $0$ is approached.\qed
\end{proof}
A generic Darboux transform, in contrast, has a limit at $0$. As above, we have
\begin{equation}
\la \ch(t)\ra=\gG^{~p}_t(\gl\of) \, \la \ch_p\ra=\la \hyfr(t)\ra \, \gG^{~p}_t(\gatr \hyfr {\gl\of}) \, \la \hyfr(p)\ra^{-1}\la\ch_p\ra.
\label{eq:expr_darb_trafo_secord2}
\end{equation}
We know that $\la \hyfr(t)\ra$ converges to $\la c(0)\basinf^*\ra$ and one might expect that if $\gG^{~p}_t(\gatr \hyfr {\gl\of}) \, \la \hyfr(p)\ra^{-1}\la\ch_p\ra$ stays sufficiently far away from the kernel $\la \basinf\ra^\perp$ of $\la c(0)\basinf^*\ra$, then in the limit $\la c(0)\basinf^*\ra$ forces $\la \ch\ra$ towards $\la c(0)\ra$. This idea is made precise and affirmed by the following Lemma, formulated in a more general context.

\begin{lemma}\label{lemma:prankone_with_arb}
For $\cV$ an arbitrary subset of $\R^m$ and $q\in\R^m$ a limit point of $\cV$, let $\la A\ra:\cV\rightarrow \Pro O(\R^{n+2}_1)$ be a map with limit $\la xy^*\ra$ at $q$, where $x,y\in\Li^{n+1}$. Let $\la u\ra:\cV\rightarrow S^n$ be a map so that there is an open neighbourhood $U$ of $\la y\ra\in S^n$ such that $\la u(\cV)\ra \cap U=\{\}$. Then the product $\la A\ra\la u\ra$ has the limit $\la x\ra$ at $q$.
\end{lemma}
\begin{proof}
Let $A$ be a lift of $\la A\ra$ with limit $xy^*$ at $q$ and $u$ the lift of $\la u\ra$ that satisfies $\ipl u(t),y\ipr=-1$ for all $t\in \cV$. Then $u$ is bounded with respect to the positive definite norm $|\cdot|$ by some $R\in\R^+$. Namely, if there was no such bound, then in order to maintain $\ipl u(t),y\ipr=-1$, the point $\la u(t)\ra$ would need to get arbitrarily close to $\la y\ra$, which contradicts the assumptions. 

Now write
\begin{equation}
A(t)u(t)=\big(A(t)-xy^*\big)u(t)+xy^*u(t)=\big(A(t)-xy^*\big)u(t)-x.
\label{eq:temp_lemma_prank1}
\end{equation}
Since $u$ is bounded by $R\in\R$, we have
$$|\big(A(t)-xy^*\big)u(t)|\leq |A(t)-xy^*||u(t)|\leq |A(t)-xy^*|R,$$
and so the first term on the right hand side of \eqref{eq:temp_lemma_prank1} has limit $0$ at $q$. Thus, the product $Au$ has the limit $-x$ at $q$ and $\la A\ra\la u\ra$ has the limit $\la x\ra$ at $q$.\qed
\end{proof}
Using this Lemma, the following proposition provides a sufficient condition for $\la \ch(t)\ra$ to converge to $\la c(0)\ra$ as $t\rightarrow 0$.

\begin{figure}
\centering
\includegraphics[width=0.75\textwidth]{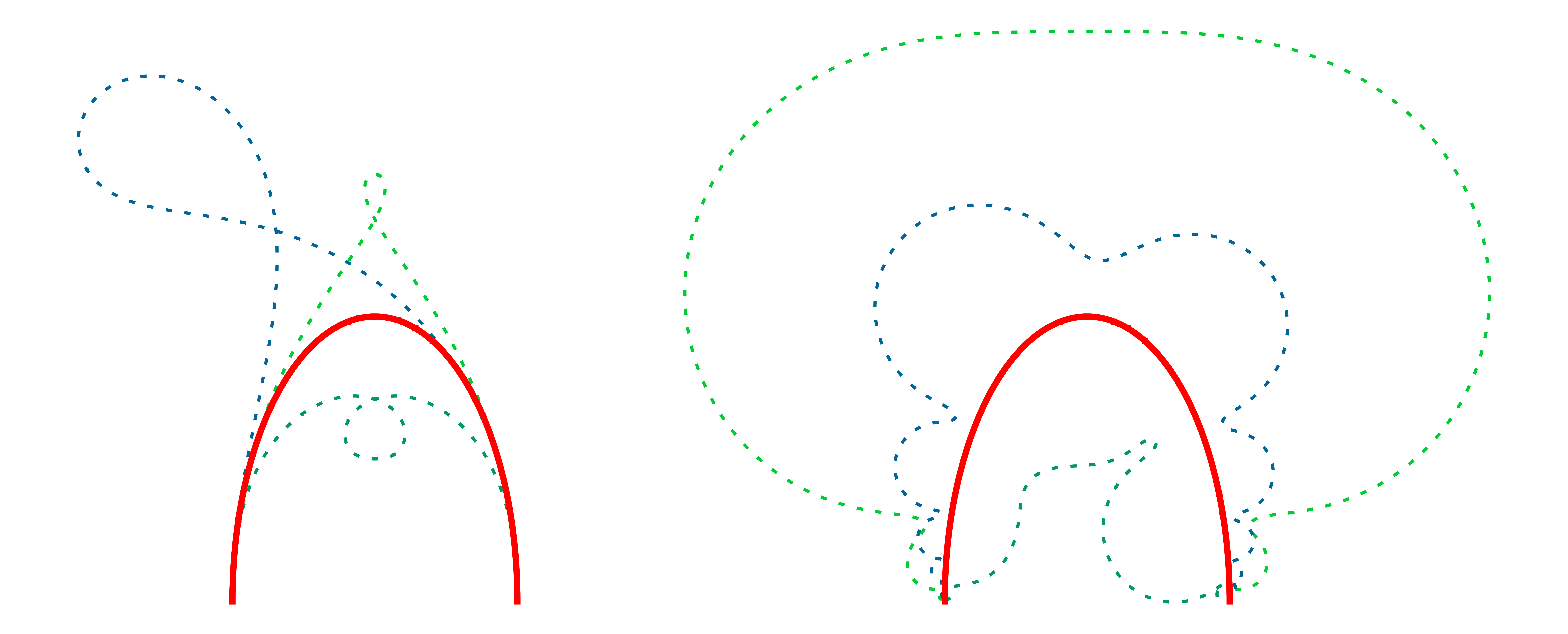}

\caption{Darboux transforms (dashed) of a half ellipse (solid) with respect to a polarization with a pole of second order at each end. The spectral parameter $\gl$ is greater (smaller) than $\frac 12$ on the right (left) side.}
\end{figure}
\begin{proposition}\label{prop:limit_secord_easy}
Let $(\la c\ra,Q)$ be a polarized curve on $(0,b)$ such that $\la c\ra$ has a regular extension to $(-\gep,b)$ and $Q$ has a pole of second order at $0$. For $\gl\in\R$ with $1-2\gl < 0$ and $\la\ch_p\ra$ not on the limit circle $\LiSe[\la c_{\gl,p}\ra]$ of the $\gl$-Calapso transform normalized at $p$ (see \eqref{eq:limit_circle_cal}), denote by $\la \ch\ra$ the $\gl$-Darboux transform of $\la c\ra$ with initial value $\la \ch(p)\ra=\la \ch_p\ra$. Then $\la \ch\ra$ converges to $\la c(0)\ra$ as $0$ is approached.
\end{proposition}
\begin{proof}
Fix $p\in (0,b)$ and factorize $\gG_p(\gatr g{\gl\of})$ as in \eqref{eq:gaugetrafo_for_int_xi} to write $\la \ch\ra$ in the form
$$\la \ch\ra=\la g\ra\la u\ra~~\text{with}~~\la u\ra:=\gG^{~p}(\xi_\gl) \, \gG^{~0}(\Gtr{\xi_\gl}p{\gatr g{\gl\of}})\la W\ra,~~~\la W\ra:=\gG_0^{~p}(\Gtr{\xi_\gl}p{\gatr g{\gl\of}}) \, \la g(p)\ra^{-1}\la \ch_p\ra.$$
The singular frame $\la g\ra$ converges to $\la c(0)\basinf^*\ra$. We want to apply Lemma \ref{lemma:prankone_with_arb} to the product $\la g\ra\la u\ra$ and thus, we need to prove that under the assumption that $\la \ch_p\ra$ does not lie on the limit circle $\LiSe[\la c_{\gl,p}\ra]$, there is a $\ti b\in (0,b)$ and neighbourhood $U\subset S^n$ of $\la \basinf\ra$ such that
\begin{equation}
\forall t\in (0,\ti b):~~~\la u(t)\ra=\gG_t^{~p}(\xi_\gl) \, \gG_t^{~0}(\Gtr{\xi_\gl}p{\gatr g{\gl\of}}) \, \la W\ra\notin U.
\label{eq:notinU}
\end{equation}
To prove this, we note that, since $\ppf_\gl$ takes values in $\la (\baszer-\bastan)\wedge (\basinf-\gl\bastan)+\R~id\ra$, we have
$$\forall w\in \Li^{n+1}~\forall t\in(0,b):~~~~\gG_{t}^{~p}(\ppf_\gl) \, \la w\ra \subset \la w,\baszer-\bastan, \basinf-\gl\bastan\ra.$$
From this we deduce the implication
\begin{equation}
\la w\ra\nsubseteq \la \baszer,\basinf,\bastan\ra~~~\Rightarrow~~~\forall t\in(0,b):~~\gG_{t}^{~p}(\ppf_\gl)\la w\ra\neq \la \basinf\ra.
\label{eq:implication_1}
\end{equation}
By the assumption of the Proposition, $\la W\ra \nsubseteq \la \baszer,\basinf,\bastan\ra$ and hence, by \eqref{eq:implication_1},
$$\forall t\in(0,b):~~~~\gG_{t}^{~p}(\ppf_\gl)\la W\ra \neq \la \basinf\ra.$$
But then there must also be open neighbourhoods $V$ of $\la W\ra$ and $U$ of $\la \basinf\ra$ such that
\begin{equation}
\forall t\in(0,b)~\forall x\in V:~~~~\gG_{t}^{~p}(\ppf_\gl)x \notin U.
\label{eq:temp_genericcase2}
\end{equation}
Due to \eqref{eq:double_limit_intbound}, we can choose $\ti b\in(0,b)$ close enough to $0$ such that
\begin{equation}
\forall t\in(0,\ti b):~~~~\gG_t^{~0}(\Gtr {\ppf_\gl}{p}{\gatr\hyfr{\gl\of}}) \, \la W\ra\in  V.
\label{eq:temp_genericcase1}
\end{equation}
Using \eqref{eq:temp_genericcase1} in \eqref{eq:temp_genericcase2} now yields \eqref{eq:notinU}. We can thus apply Lemma \ref{lemma:prankone_with_arb} to the product $\la \ch\ra=\la g\ra\la u\ra$ restricted to $(0,\ti b)$. Since $p\in (0,b)$ was arbitrary, this completes the proof.\qed
\end{proof}

The characterization of the $\gl$-Darboux transforms excluded in Prop \ref{prop:limit_secord_easy} is independent of the point $p\in (0,b)$, at which the initial condition $\la \ch(p)\ra=\la \ch_p\ra$ is posed. As soon as $\la \ch(p)\ra \in \LiSe[\la c_\gl,p\ra]$ holds for one $p\in (0,b)$, it holds for all $p\in (0,b)$.

We finally analyse what happens, when the initial point of the Darboux transform does lie on the limit circle of the Calapso transform. In this case, we cannot apply Lemma \ref{lemma:prankone_with_arb} and indeed $\la \ch\ra$ does not converge to $\la c(0)\ra$ as $0$ is approached.

\begin{theorem}\label{theorem:darb_tr_sec_ord_pos}
Let $(\la c\ra,Q)$ be a polarized curve on $(0,b)$ where $\la c\ra$ has a regular extension to $(-\gep,b)$ and $Q$ has a pole of second order at $0$. As $0$ is approached, any $\gl$-Darboux transform $\la \ch\ra$ with $1-2\gl<0$ converges to $\la c(0)\ra$, 
$$\lim_{t\rightarrow 0}\la \ch(t)\ra=\la c(0)\ra,$$
except when $\la \ch(p)\ra$ lies on the limit circle $\LiSe[\la c_{\gl,p}\ra]$ of the Calapso transform $\la c_{\gl,p}\ra$. In that case $\la \ch\ra$ approaches the curvature circle of $\la c\ra$ at $0$ rotating with frequency increasing towards infinity as $0$ is approached.
\end{theorem}
\begin{proof}
The first part of the theorem was proved above (see Prop \ref{prop:limit_secord_easy}). What remains to be confirmed is the exception. The analysis of this exceptional case is laborious and will occupy the rest of this section. We prove the theorem in three steps. First, we show that $\la \ch\ra$ approaches the curvature circle of $\la c\ra$ at $0$. Then we show that $\la \ch\ra$ has at least two distinct limit points on that curvature circle at $0$. In the third step, we use that $\la \ch\ra$ and $\la c\ra$ envelop a circle congruence to argue that $\la \ch\ra$ tends towards a rotation on the curvature circle of $\la c\ra$ at $0$ with frequency increasing towards infinity as $0$ is approached.

Throughout this proof we consider $p\in (0,b)$ and $\gl\in\R$ with $1-2\gl<0$ fixed. To simplify notation, we define the integrable 1-form $\phi$ and its orthogonal lift $\Phi$ by
$$\phi:=\Gtr {\ppf_\gl}{p}{\gatr\hyfr{\gl\of}},~~~~~~~~\Phi:=\Gtr {\Ppf_\gl}{p}{\gatr\hyfr{\gl\Of}}.$$
In particular, from \eqref{eq:gauge_trafo_eta} and \eqref{eq:gaugetrafo_poleformfirstkind} we find
$$\Phi_t=\Phi(t)\dd t=\gG_p^{~t}(\Ppf_\gl)~t~\baszer\wedge \Fr^\perp_t~\gG_t^{~p}(\Ppf_\gl),$$
such that $\Phi_t/\id$ is bounded on $(0,b)$.

Let $\la \ch_p\ra$ be any point on $\LiSe[\la c_{\gl,p}\ra]$, that is (cf Thm \ref{theorem:cal_tra_at_b_secord_pos}),
\begin{equation}
\la\ch_p\ra=\la \hyfr(p)\ra\gG_{p}^{~0}(\phi)\la\Ch_p\ra~~~~\text{with}~~~\Ch_p\in\la \baszer,\basinf,\bastan\ra.
\label{eq:temp_lastprove1}
\end{equation}
Then substitution of \eqref{eq:temp_lastprove1} in \eqref{eq:expr_darb_trafo_secord2} and using $\gG^{~p}(\gatr g{\gl\of})=\gG^{~p}(\xi_\gl)\gG^{~p}(\phi)$ yields
\begin{equation}
\la \ch\ra=\la g\ra \, \gG^{~p}(\gatr \hyfr {\gl\of}) \, \gG_{p}^{~0}(\phi) \, \la\Ch_p\ra=\la g\ra \, \gG^{~p}(\ppf_\gl) \, \gG^{~0}(\phi) \, \la\Ch_p\ra.
\label{eq:ch_in_lastproof}
\end{equation}

\begin{enumerate}
	\item Denote by $\Ppf_\gl$ the orthogonal lift of $\ppf_\gl$ and let $\ch$ be the lift
	\begin{equation}
	\ch(t)=\frac{1}{\id}\hyfr(t) \, \gG_t^{~p}(\Ppf_\gl) \, \gG_t^{~0}(\Phi) \, \Ch_p
	\label{eq:lift_bounded_from_below}
	\end{equation}
	for some nonzero $\Ch_p\in\la \Ch_p\ra$. In Sect \ref{section:sing_gauge}, we chose the normal fields $\{N_1,...,N_{n-1}\}$ such that $\la N_1,...,N_{n-1}\ra^\perp$ is the curvature circle congruence of $\la c\ra$. Therefore, since $\ch$ is a lift of $\la \ch\ra$ that stays away from zero, in order to show that $\la \ch\ra$ approaches the curvature circle of $\la c\ra$ at $0$, it is sufficient to show that the $n-1$ inner products $\ipl \ch,N_i \ipr$ all tend towards zero. Since $\hyfr^{-1}N_i=\basnor_i$ and $\Ppf_\gl(\basnor_i)=0$ (see Sect \ref{section:sing_gauge}), the inner product of \eqref{eq:lift_bounded_from_below} with $N_i$ reads
	$$\ipl \ch(t),N_i(t) \ipr=\frac{1}{\id}\Ipl\gG_t^{~0}(\Phi) \, \Ch_p,\basnor_i\Ipr=:\frac 1 {\id}\gs_i(t).$$
	Since $\Ch_p$ is an element of $\la \baszer,\basinf,\bastan\ra$ and $\basnor_i\in \la \baszer,\basinf,\bastan\ra^\perp$, all the functions $\gs_i(t)$ tend to zero as $t\rightarrow 0$. Since $\Phi(t)/t$ is bounded, $\Phi(t)$ tends to zero as $0$ is approached and thus also the first derivatives of all $\gs_i$ tend to zero as $0$ is approached. An application of the mean value theorem to $\gs_i$ then yields that $\ipl \ch,N_i \ipr$ tends to zero and indeed $\la \ch\ra$ approaches the curvature circle of $\la c\ra$ at $0$ as $0$ is approached.
	
	\item We now prove that $\la \ch\ra$ has at least two limit points at $0$, namely $\la c(0)\ra$ and $\la \bar c(0)\ra$, where $\bar c$ is as in Sect \ref{section:sing_gauge}.
	
	To prove that $\la c(0)\ra$ is a limit point, let $(t_i)_{i\in\mathbb N}$ be a sequence in $(0,b)$ which converges to $0$ and is such that
	$$\forall i\in\mathbb N:~\gG_{t_i}^{~p}(\ppf_\gl) \, \la\Ch_p\ra= \la \baszer\ra.$$
	In particular 
	$$\forall i,j>0:~\gG_{t_i}^{~t_j}(\ppf_\gl)=id.$$
	Then, evaluating \eqref{eq:ch_in_lastproof} at $t_i$ and taking the limit $i\rightarrow\infty$ yields
	$$	\lim_{i\rightarrow \infty}\la \ch(t_i)\ra=\lim_{i\rightarrow \infty}\la \hyfr(t_i)\ra\gG_{t_i}^{~p}(\ppf_\gl)\gG_{t_i}^{~0}(\phi)\la \Ch_p\ra=\la c(0)\basinf^*\ra\gG_{t_1}^{~0}(\ppf_\gl)\la \Ch_p\ra=\la c(0)\basinf^*\ra\la \baszer\ra=\la c(0)\ra.$$
	Thus, $\la c(0)\ra$ is a limit point of $\la \ch\ra$.
	
	To prove that also $\la \bar c(0)\ra$ is a limit point, rescale $\Ch_p$ if necessary and let $(t_i)_{i\mathbb N}$ be a sequence in $(0,b)$ which converges to $0$ and is such that
	\begin{equation}
	\forall i\in\mathbb N:~\gG_{t_i}^{~p}(\Ppf_\gl) \Ch_p=  \basinf.
	\label{eq:secondsequence_excepcase}
	\end{equation}
		The lift \eqref{eq:lift_bounded_from_below} evaluated at $t_i$ yields the sequence $(\ch(t_i))_{i\in\mathbb N}$ with
	$$\ch(t_i)=\frac{1}{t_i}\hyfr(t_i) \, \gG_{t_1}^{~p}(\Ppf_\gl) \, \gG_{t_i}^{~0}(\Phi) \, \ch_p.$$
	We have already proved above that the $n-1$ quantities $\ipl \ch,N_i \ipr$ converge to zero. We now show that also both $\ipl \ch(t_i),\bar c(t_i)\ipr$ and $\ipl \ch(t_i),\pr c(t_i)/\|\pr c(t_i)\|\ipr$ converge to zero and that $\ipl \ch(t_i),c(t_i)\ipr$ converges to $-1$. It then follows that $\la \ch(t_i)\ra$ converges to $\la \bar c(0)\ra$.
	
	Use $g^{-1}(t)\bar c(t)=\frac 1 t \basinf$ and \eqref{eq:secondsequence_excepcase} to write
	$$\ipl \ch(t_i),\bar c(t_i)\ipr =\frac{1}{t_i^2}\Ipl\gG_{t_i}^{~p}(\Ppf_\gl) \, \gG_{t_i}^{~0}(\Phi) \, \Ch_p,\basinf\Ipr=\frac{1}{t_i^2}\Ipl\gG_{t_i}^{~0}(\Phi) \, \Ch_p,\Ch_p\Ipr.$$
	Now use \eqref{eq:integral_eq} for $\gG_{t}^{~p}(\Phi)$ twice and the antisymmetry of elements of $\ort(\Rmn)$ to obtain
	$$\ipl\ch(t_i),\bar c(t_i)\ipr=\frac{1}{t_i^2}\int^0_{t_i}\int^0_{\tau}\Ipl\Phi(\tau) \, \Phi(\ti\tau) \, \gG_{\ti\tau}^{~0}(\Phi) \, \Ch_p,\Ch_p\Ipr \dd\ti\tau\dd\tau.$$
	Since $\Phi(t)/t$ is bounded, indeed $\ipl \ch(t_i),\bar c(t_i)\ipr$ converges to zero as $i\rightarrow \infty$.
	
	To show that the inner product of $\ch(t_i)$ with $\pr c(t_i)/\|\pr c(t_i)\|$ converges to zero, we write
	$$\ipl \ch(t_i),\pr c(t_i)/\|\pr c(t_i)\|\ipr=\frac{1}{t_i}\Ipl\gG_{t_i}^{~p}(\Ppf_\gl) \, \gG_{t_i}^{~0}(\Phi) \, \Ch_p,\bastan\Ipr.$$
	Again use the integral equation for $\gG^{~0}(\Phi)$, the boundedness of $\Phi/\id$ and \eqref{eq:secondsequence_excepcase} to conclude that $\ipl\ch(t_i),\pr c(t_i)/\|\pr c(t_i)\|\ipr$ converges to zero as $i\rightarrow\infty$.
	
	Finally,
	$$\ipl\ch(t_i),c(t_i)\ipr=\Ipl\gG_{t_i}^{~p}(\Ppf_\gl) \, \gG_{t_i}^{~0}(\Phi) \, \Ch_p,\baszer\Ipr$$
	which converges to $\ipl \gG_{t_1}^{~p}(\Ppf_\gl)\Ch_p,\baszer\ipr=\ipl \basinf,\baszer\ipr=-1$.
	
	Thus, indeed $\lim_{i\rightarrow\infty}\ch(t_i)=\bar c(0)$ such that $\la \bar c(0)\ra$ is also a limit point of $\la \ch\ra$.

\item We now show that $\la \ch\ra$ tends towards a rotation on the curvature circle of $\la c\ra$ at $0$. Since both $\la c(0)\ra$ and $\la \bar c(0)\ra$ are limit points of $\la \ch\ra$, it then follows that the frequency of this rotation must tend towards infinity.

Choose $\ti b\in (0,b)$ such that $\la \bar c(0)\ra\notin \la c((0,\ti b))\ra$. Then there is a closed neighbourhood $U\subset S^n$ of $\la \bar c(0)\ra$ whose intersection with $\la c((0,\ti b))\ra$ is empty. Define 
$$\I:=\{t\in (0,\ti b)~|~\la \ch(t)\ra\in U\}.$$
Since $\la \bar c(0)\ra$ is a limit point of $\la \ch\ra$, the set $\I$ is non-empty and has the limit point $0$. Now let $\bar c$, $c$ and $\ch$ be spherical lifts of $\la \bar c\ra$, $\la c\ra$ and $\la \ch\ra$, respectively. The circle congruence $\cS$ enveloped by $\la \ch\ra$ and $\la c\ra$ is 
$$\cS=\la \ch,c,\pr c\ra=\La \frac{1}{\ipl c,\ch\ipr}\Pi_{\la N_1,...,N_{n-1}\ra}\ch-\bar c,c,\pr c\Ra.$$
We have showed above that $\Pi_{\la N_1,...,N_{n-1}\ra}\ch$ converges to zero as $0$ is approached. Moreover, we have constructed $\I$ in a way that there exists a constant $C>0$ such that
\begin{equation}
\forall t\in\I:~~~~\ipl c(t),\ch(t)\ipr>C.
\label{eq:lowerboundonU1}
\end{equation}
Therefore, $\cS$ restricted to $\I$ converges to the curvature circle $\la \bar c(0),c(0),\pr c(0)\ra$ of $\la c\ra$ at $0$. \eqref{eq:lowerboundonU1} also implies that inside $U$ the speed of $\la \ch\ra$ tends towards infinity because 
\begin{equation}
\forall t\in\I:~~~~\|\pr \ch(t)\|^2=\|\gl \Qend\,(c\wedge \pr c)(\ch)\|^2=\frac{\gl^2}{t^2}\ipl c(t),\ch(t)\ipr^2\|\pr c\|^2>\frac{\gl^2}{t^2}C^2\|\pr c\|^2.
\label{eq:lowerboundonU2}
\end{equation}
Now stereographically project $S^n$ to $\R^n$ such that the curvature circle of $\la c\ra$ at $0$ gets mapped to the unit circle in the $\la e_1,e_2\ra$-plane with centre the origin. Write the projection $\chfr$ of $\la \ch\ra$ as 
$$\chfr(t)=\rho(t)\left(\sin(\vphi(t))e_1+\cos(\vphi(t))e_2\right)+Z(t),$$
where $\vphi:(0,b)\rightarrow \R$ is smooth and $Z(t)\in \la e_3,...,e_n\ra$ for all $t\in (0,b)$. Since $\la \ch\ra$ approaches the curvature circle of $\la c\ra$ at $0$, the function $Z$ must converge to zero and $\rho$ must converge to $1$ as $t\rightarrow 0$. Since moreover the enveloped circle congruence $\cS$ restricted to $\I$ converges to the curvature circle of $\la c\ra$ at $0$, on $\I$ the angle between the tangent to $\chfr$ and the tangent to the circle $\sin(\vphi)e_1+\cos(\vphi)e_2$ must converge to $\pm 1$, that is,
$$\pm 1=\lim_{\overset{t\rightarrow 0}{t\in \I}}\Ipl \frac{\pr\chfr(t)}{\|\pr\chfr(t)\|},\cos(\vphi(t))e_1-\sin(\vphi(t))e_2\Ipr=\lim_{\overset{t\rightarrow 0}{t\in \I}}\frac{\rho(t)\pr\vphi(t)}{\sqrt{\pr\rho(t)^2+\rho^2\pr\vphi(t)^2+\|\pr Z(t)\|^2}}.$$
Due to the lower bound \eqref{eq:lowerboundonU2}, the square root in this expression tends to infinity and so also $\pr\vphi$ must tend to plus or minus infinity as $0$ is approached in $\I$. Together with the result that also $\la c(0)\ra$ is a limit point of $\la \ch\ra$, this implies that $\vphi$ tends to plus or minus infinity such that indeed the frequency of the rotation of $\la \ch\ra$ on the curvature circle tends towards infinity.

\end{enumerate}
\qed
\end{proof}

\begin{figure}
\centering
\includegraphics[width=0.75\textwidth]{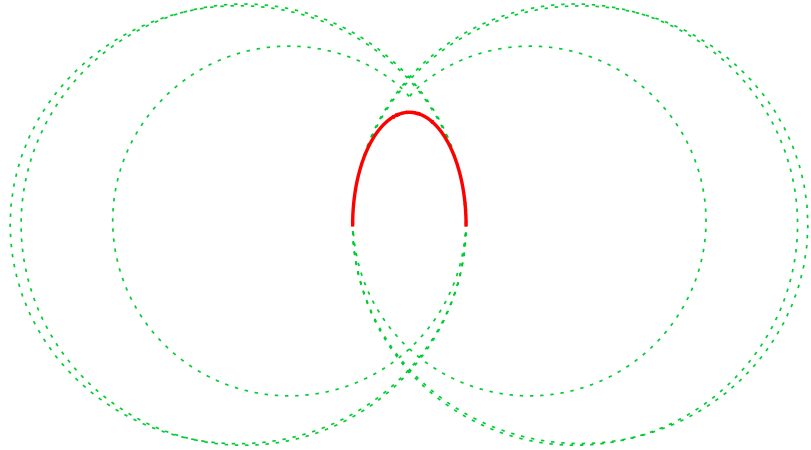}

\caption{Darboux transform (dashed) of a half ellipse (solid) with respect to a polarization with a pole of second order at each end. The spectral parameter $\gl$ is positive and the initial point $\la \ch(p)\ra$ lies on the limit circle of the Calapso transform normalised at $p$.}
\end{figure}

\bibliographystyle{spmpsci}      % mathematics and physical sciences
\bibliography{bib}

\end{document}